\documentclass[smallextended,numbook,envcountsect,runningheads,referee,letterpaper]{svjour3}

\newif\ifIsJCP
\newif\ifIsSpringer

\makeatletter
\@ifclassloaded{elsarticle}{\IsJCPtrue}{}
\@ifclassloaded{svjour3}{\IsSpringertrue}{}
\makeatother

\usepackage{pdf14}

\usepackage[T1]{fontenc}
\usepackage[bitstream-charter]{mathdesign} 
\usepackage[scaled]{berasans}              

\newlength{\alphabet}
\usepackage{geometry}
\ifIsJCP
\else
\fi

\usepackage{environ}
\newif\ifIsArxiv
\IsArxivtrue
\ifIsArxiv
\def\makeheadbox{}
\renewcommand\date[1]{}
\NewEnviron{alignln}{%
\begin{align}
\BODY
\end{align}}
\else
\usepackage[modulo]{lineno}
\linenumbers
\NewEnviron{alignln}{%
\begin{linenomath*}
\begin{align}
\BODY
\end{align}
\end{linenomath*}}
\fi

\ifIsSpringer
\smartqed
\journalname{BIT}
\fi

\usepackage[fleqn,reqno]{amsmath} 
\usepackage{mathtools}
\usepackage{bm}
\DeclareMathAlphabet{\mathpzc}{OT1}{pzc}{m}{it} 
\numberwithin{equation}{section}

\usepackage{xspace}                   
\usepackage{xifthen}
\usepackage{ifdraft}
\usepackage{paralist}                 
\usepackage{lipsum}                   
\usepackage{soul}                     
\usepackage[notref,final]{showkeys}

\usepackage{siunitx}                 
\usepackage[usenames,dvipsnames,svgnames,table]{xcolor}

\newcommand\citet[1]{\cite{#1}}

\usepackage{hyperref}
\hypersetup{
  colorlinks  = true,
  urlcolor    = blue,
  linkcolor   = blue,
  citecolor   = blue,
  colorlinks  = true,
  pdfauthor   = {Abtin Rahimian},
  pdftitle    = {Write up},
  pdfcreator  = {pdflatex},
  pdfpagemode = UseNone
}

\usepackage{relsize}
\def\abbrevsize{.85}
\newcommand\abbrevformat[1]{\textscale{\abbrevsize}{#1}}

\newcommand\abbrev[1]{\abbrevformat{#1}}

\makeatletter
\newcommand{\hypertargetraised}[1]{\Hy@raisedlink{\hypertarget{#1}{}}}
\makeatother

\usepackage[capitalize]{cleveref}
\crefname{equation}{Eq.}{Eqs.}
\Crefname{equation}{Equation}{Equations}
\crefrangelabelformat{equation}{(#3#1#4--#5#2#6)}
\crefmultiformat{equation}{Eqs. (#2#1#3}{, #2#1#3)}{#2#1#3}{#2#1#3}
\Crefmultiformat{equation}{Equations (#2#1#3}{, #2#1#3)}{#2#1#3}{#2#1#3}
\crefname{apx}{Appendix}{Appendices}
\crefname{algocf}{Algorithm}{Algorithms} 

\usepackage[title,titletoc,toc]{appendix}

\ifIsJCP
\usepackage{amsthm}
\newtheorem{theorem}{Theorem}[section]

\newtheorem{remark}[theorem]{Remark}
\fi 

\newcommand\para[1]{\paragraph{\textbf{#1}.}}
\newcommand\lhb[1]{\textbf{#1}}        

\usepackage{tabularx}     
\usepackage{booktabs}     
\usepackage{multirow}
\usepackage{collcell}     
\usepackage{dcolumn}      

\usepackage[linesnumbered,lined,ruled,vlined,commentsnumbered]{algorithm2e}

\usepackage[margin=10pt,font=small,textfont=it,labelfont=bf,justification=justified]{caption}
\newcommand{\algcaption}[3]{
        \ifthenelse{\isempty{#3}}
                   {\caption[#1]{{\sc #2.} \label{#1}}}
                   {\caption[#1]{{\sc #2.} \newline\small{#3} \label{#1}}}
        }

\usepackage{graphicx}
\graphicspath{{figs/}}
\usepackage[export]{adjustbox}  
\usepackage[
  format=hang,
  singlelinecheck=false,
  font={scriptsize},
  labelfont=bf,
  labelformat=simple,
  subrefformat=simple,
  justification=centering
]{subfig}                      
\newcommand{\mcaption}[3]{
  \ifthenelse{\isempty{#2}}
             {\caption[#1]{#3 \label{#1}}}
             {\caption[#1]{{\sc #2.} #3 \label{#1}}}
}

\newlength\figureheight
\newlength\figurewidth

\newif\ifUseTikz
\ifUseTikz
\usepackage{tikz,pgfplots,import}
\pgfplotsset{compat=newest}

\usetikzlibrary{pgfplots.groupplots}
\usetikzlibrary{calc}
\usetikzlibrary{plotmarks}
\usetikzlibrary{arrows.meta}
\usetikzlibrary{backgrounds}
\usetikzlibrary{external}

\pgfplotsset{plot coordinates/math parser=false}

\pgfkeys{/pgf/images/include external/.code={\includegraphics[]{#1}}}
\pgfkeys{/tikz/external/mode=list and make}
\newcommand{\includepgf}[2][1]{
  \beginpgfgraphicnamed{#2}%
  \tikzsetnextfilename{tikz-external-#2}%
  \scalebox{#1}{\subimport{figs/}{#2.pgf}}%
  \endpgfgraphicnamed%
}

\makeatletter
\tikzset{
    every picture/.style={
        execute at begin picture={
            \let\ref\@refstar
        }
    }
}
\makeatother

\else 
\newcommand{\includepgf}[2][1]{
\scalebox{#1}{\includegraphics[]{tikz-external-#2}}%
}
\fi

\usepackage[normal]{engord} 

\let\d\undefined
\let\O\undefined
\DeclareMathOperator{\Grad}{\nabla}
\DeclareMathOperator{\Div}{\nabla\cdot}

\newcommand\defeq{\mathrel{\mathop :}=}

\newcommand\inner[2]{#1\cdot #2}

\newcommand\d{\ensuremath{\,\mathrm{d}}}

\newcommand\parderiv[2]{\frac{\partial #1}{\partial #2}}

\newcommand\ii{i}
\newcommand\O[1]{\ensuremath{\mathcal{O}\left(#1\right)}}
\newcommand\ordinal[1]{\ifthenelse{\isin{#1}{abcdefghijklmnopqrstuvwxyz}}{\ensuremath{#1^\mathrm{th}}}{\engordnumber{#1}}}
\newcommand\R{\mathbb{R}}

\let\vector\undefined
\DeclareMathAlphabet{\mathbfsf}{\encodingdefault}{\sfdefault}{bx}{n}
\newcommand\discrete[1]{\mathsf{#1}}

\newcommand\scalar[1]  {#1}
\newcommand\scalard[1] {\discrete{#1}}
\newcommand\vector[1]  {\bm{#1}}
\newcommand\vectord[1] {\bm{\mathsf{#1}}}

\newcommand\conv[1]    {\mathcal{#1}}

\renewcommand\Im{\operatorname{Im}}
\def\reynolds{\ensuremath{\mathit{R\kern-.13em e}}}  
\def\capillary{\ensuremath{\mathit{C\kern-.21em a}}} 

\newcommand\vx{{\vector{x}}}  
\newcommand\vX{{\vector{X}}}
\newcommand\vy{{\vector{y}}}
\newcommand\vz{{\vector{z}}}
\newcommand\vn{{\vector{n}}}
\newcommand\vr{{\vector{r}}}
\newcommand\vu{{\vector{u}}}

\def\dl{double-layer\xspace}   
\def\ns{near-singular\xspace}

\def\gl{Gauss--Legendre\xspace}
\def\nystrom{Nystr\"om\xspace}
\def\sl{single-layer\xspace}

\newcommand\op[1]{\mathcal #1}

\def\emdash/{\kern 0.2em---\kern 0.2em}
\def\qbkix{\abbrevformat{QBKIX}\xspace}
\def\qbkixlong{Quadrature by Kernel-Independent Expansion\xspace}
\def\qbx{\abbrevformat{QBX}\xspace}
\def\gmres{\abbrevformat{GMRES}\xspace}
\def\pde{\abbrevformat{PDE}\xspace}
\def\qbxlong{Quadrature by Expansion\xspace}
\def\ki{kernel-independent\xspace}
\def\kifmm{\abbrevformat{KIFMM}\xspace}
\def\rc{\ensuremath{{r_c}}}
\def\nc{\ensuremath{{n_c}}}
\def\rp{\ensuremath{R}}
\def\np{\ensuremath{{n_p}}}
\def\re{\ensuremath{{r}}}
\def\remax{\ensuremath{\delta}}
\def\cen{\ensuremath{\vector{c}}}
\def\acc{\ensuremath{\varepsilon}}
\def\rtol{\ensuremath{\acc_r}}
\def\atol{\ensuremath{\acc_a}}
\def\apinv{\acc_\mathrm{pinv}}
\def\ech{\ensuremath{e_c}}
\def\err{\ensuremath{e}}
\newcommand\e[1]{\ensuremath{e{#1}}}

\newcommand\sci[2][1]{\ensuremath{\ifthenelse{\equal{#1}{1}}{}{#1\times}10^{#2}}}
\newcommand\twod{\abbrev{2D}\xspace}
\newcommand\threed{\abbrev{3D}\xspace}

\renewcommand\u{\scalar{u}}   
\newcommand\uq {\hat   {u}}   
\newcommand\un {\tilde {\u}}  
\newcommand\unv{\tilde {\vu}} 

\newcommand\fundsol{\Phi}

\newcommand{\bi}{\begin{itemize}}
\newcommand{\ei}{\end{itemize}}
\newcommand{\ben}{\begin{enumerate}}
\newcommand{\een}{\end{enumerate}}
\newcommand{\be}{\begin{equation}}
\newcommand{\ee}{\end{equation}}
\newcommand{\bea}{\begin{eqnarray}}
\newcommand{\eea}{\end{eqnarray}}
\newcommand{\ba}{\begin{align}}
\newcommand{\ea}{\end{align}}
\newcommand{\bse}{\begin{subequations}}
\newcommand{\ese}{\end{subequations}}
\newcommand{\bfi}{\begin{figure}}
\newcommand{\efi}{\end{figure}}

\newcommand{\half}{\mbox{\small $\frac{1}{2}$}}
\newcommand{\xx}{\vx}       
\newcommand{\yy}{\vy}
\newcommand{\zz}{\vz}

\newcommand{\cc}{{\vector{c}}}
\newcommand{\RR}{\mathbb{R}}
\newcommand{\bal}{\bm{\alpha}}

\newcommand\disc[3]{#1_{#2}^{#3}}
\newcommand\proxyc{\disc{\partial B}{\rp}{\cen}}
\newcommand\checkc{\disc{\partial B}{\rc}{\cen}}
\newcommand\evald {\disc{B}{\remax}{\cen}}

\begin{document}

\title{Ubiquitous evaluation of layer potentials using
  \qbkixlong}

\ifIsSpringer
\titlerunning{Ubiquitous evaluation of layer potentials using \qbkix}

\author{Abtin Rahimian \and
  Alex Barnett         \and
  Denis Zorin
}

\institute{ Abtin Rahimian \and Denis Zorin \at
  Courant Institute of Mathematical Sciences, New York
  University, New York, NY 10003
  \email{arahimian@acm.org}, \email{dzorin@cims.nyu.edu} \\
  Alex Barnett \at
  Department of Mathematics, Dartmouth College, Hanover, NH 03755
  \email{ahb@math.dartmouth.edu}
}

\date{Received: (date) / Accepted: (date)}
\fi
\ifIsJCP
\title{\qbkixlong\\ for Evaluation of Layer Potentials}
\author[nyu]{Abtin Rahimian} \ead{arahimian@acm.org}
\author[dart]{Alex Barnett} \ead{ahb@math.dartmouth.edu}
\author[nyu]{Denis Zorin} \ead{dzorin@cims.nyu.edu}
\address[nyu]{Courant Institute of Mathematical Sciences, New York
  University, New York, NY 10003}
\address[dart]{Department of Mathematics, Dartmouth College, Hanover,
  NH 03755}
\fi

\maketitle

\begin{abstract}
  We introduce a quadrature scheme\emdash/\qbkix{}\emdash/for the
  high-order accurate evaluation of layer potentials associated with
  general elliptic {\pde}s near to and on the domain boundary.
  Relying solely on point evaluations of the underlying kernel, our
  scheme is essentially \pde-independent; in particular, no analytic
  expansion nor addition theorem is required.
  Moreover, it applies to boundary integrals with singular, weakly
  singular, and hypersingular kernels.

  Our work builds upon Quadrature by Expansion (\qbx),
  which approximates the potential by an analytic expansion
  in the neighborhood of each expansion center.
  In contrast, we use a sum of fundamental solutions lying on a ring
  enclosing the neighborhood, and solve a small dense linear system
  for their coefficients to match the potential on a smaller
  concentric ring.

  We test the new method with Laplace, Helmholtz, Yukawa, Stokes, and
  Navier (elastostatic) kernels in two dimensions (\twod) using adaptive,
  panel-based boundary quadratures on smooth and corner domains.
  Advantages of the algorithm include its relative simplicity of
  implementation, immediate extension to new kernels,
  dimension-independence (allowing simple generalization to \threed),
  and compatibility with fast algorithms such as the
  kernel-independent \abbrevformat{FMM}.

  %
\end{abstract}

\section{Introduction\label{sec:intro}}
The boundary integral method is a powerful tool for solving linear
partial differential equations ({\pde}s) of classical physics with
piecewise constant material coefficients, with applications including
electromagnetic scattering, molecular electrostatics, viscous fluid
flow, and acoustics.  It involves exploiting Green's theorems to
express the solution in terms of an unknown ``density'' function
defined on the domain boundaries or material interfaces, using the
physical boundary condition to formulate an integral equation for this
density, and finally obtaining a linear algebraic system via Galerkin,
\nystrom, or other discretization.  Compared to commonly used
differential formulations, boundary integral methods have a number of
advantages: decreasing the dimension of the problem that needs to be
discretized, avoiding meshing the volume, and improving conditioning.  For
instance, the integral equation can often be chosen to be a Fredholm
equation of the second kind, resulting in a well-conditioned linear
system which can be solved by a Krylov subspace
methods in a few iterations. All these considerations are particularly important for problems with
complicated and moving geometries
\cite{helsing_close,rahimian2010b,corona2015,ojalastokes}.

The main difficulty in using boundary integral methods is the
need to evaluate singular and nearly-singular integrals:
\begin{inparaenum}[(i)]
\item Evaluating system matrix entries requires
  evaluation of the potential on the surface, which involves a
  \emph{singular} integral;
\item Once the density is solved for, the desired solution must
  still be evaluated in the form of a potential.
  As an evaluation point approaches the boundary of the domain, the
  peak in the resulting integrand becomes taller and narrower, giving
  rise to what is referred to as a \emph{\ns} integral.
  The result is an arbitrarily high loss of accuracy, if the distance
  from points to the surface is not bounded from below, when a quadrature scheme
  designed for smooth integrands is used \cite[Section~7.2.1]{atkinson} and \cite{ce}.
\end{inparaenum}

\begin{figure}[!tb]
  \subfloat[]{\includegraphics[width=3in,height=1.4in]{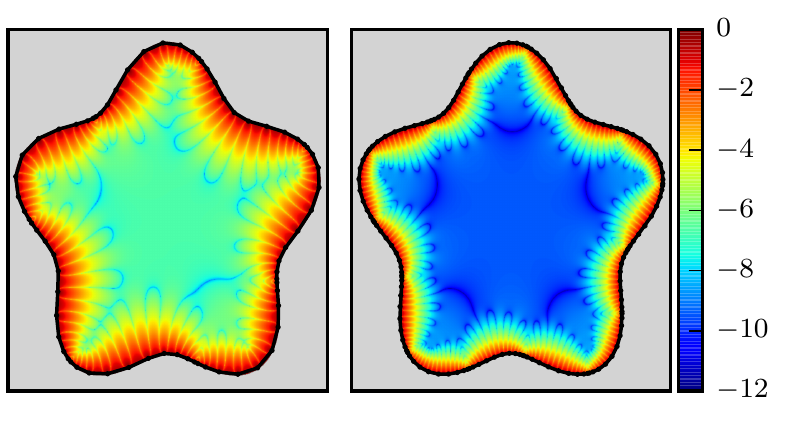}}
  \subfloat[]{\includegraphics[width=3in,height=1.4in]{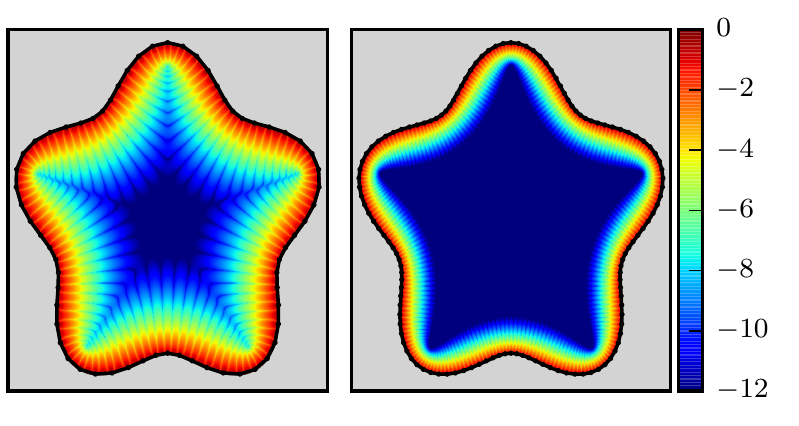}}
  \mcaption{fig:near-sing}{}{Evaluation error plotted in the solution
    domain due to approximating the Laplace \dl potential
    \cref{lapDLP} using a quadrature designed for smooth
    functions.
    Logarithm of absolute error, $\log_{10} |\un(\vx)-\u(\vx)|$, where
    $\u$ is the true solution and $\un$ is the discrete approximation
    using smooth quadrature is plotted for the case of constant
    density $\phi\equiv 1$.
    (a)~shows composite quadrature with $M=7$ (left) or $M=15$ (right)
    panels each with $q=10$ \gl nodes.
    (b)~shows the global composite trapezoid rule with $N=64$ (left) or
    $N=128$ (right) nodes.
  }
\end{figure}
\Cref{fig:near-sing} illustrates the near-singular evaluation of the solution
$u$ of the Dirichlet Laplace equation in a simple smooth domain, which
is represented by the double-layer potential
\begin{equation}
  u(\vx) = \frac{1}{2\pi} \int_\Gamma \frac{\partial}{\partial
    \vn_\vy} \log \frac{1}{\|\vx-\vy\|}\cdot\phi(\vy) \d s(\vy) ~,
\label{lapDLP}
\end{equation}
where $\phi$ is the density defined on the boundary $\Gamma$.
The growth in error as $\vx$ approaches $\Gamma$ is apparent in all
four plots (showing panel-based and global quadratures with different
numbers of nodes $N$).
Although the width of the  high-error layer near the boundary shrinks like $1/N$
\cite{ce}, the error always reaches \O{1} at the boundary.
The goal of this paper is to present a flexible scheme that
handles both tasks (singular and near-singular evaluation) to high-order accuracy in a
kernel-independent (i.e., {\pde}-independent) manner.

\para{Related work}
Designing quadrature schemes for singular and \ns integrals has a long
and rich history \cite{atkinson,LIE}.
Until recently, the quadrature methods were designed specifically for
either on-surface evaluation or near-surface evaluation.
Many of the on-surface integration quadrature are specific to a
certain type of kernel~(singularity), e.g., $\log |\vr|$ in \twod or
$1/|\vr|$ in \threed
\cite{kapur,alpert,helsing,kolm2001,kress95,sidi1988,strain1995,yarvin1998,bremer3d};
the former case is reviewed in \cite{hao}.

A popular method for on-surface quadrature is the product integration
(in \twod, for the global trapezoid rule see \cite[Section~4.2]{atkinson}
or \cite[Section~12.3]{LIE}, and for panel-based rules see
\cite{helsing_close}).
In this context, an analytic convolution of the kernel with each function in some
basis set is found, reducing evaluation of the integral to projection
of the boundary density onto that basis set.

Another approach for on-surface evaluation is singularity
subtraction, where the integrand is modified by subtracting an
expression that eliminates its singularity
\cite[Chapter~2]{davisrabin} and \cite{pozrikidis,jarvenpaa2003}.
However, this leaves high-order singularities in the kernel which
makes the higher derivatives of the kernels unbounded, limiting the
accuracy of the quadrature scheme.
Alternatively, for weakly singular kernels, one can use
transformations to cancel the singularity by the decay of area element
(e.g., in \threed using Duffy transformation
\cite{duffy1982} or polar coordinates)
\cite{brunoFMM,graglia2008,khayat2005,schwab1992,ying06,farina2001,johnson1989,Veerapaneni2011,Ganesh2004,Graham2002}.
To achieve a high convergence order, these methods need some form of partition of
unity so that a high-order polar patch can be constructed around each
point \cite{ying06}.

One can also regularize the kernel and then exploit quadrature schemes for
smooth functions \cite{lowengrub1993,Tornberg2004}.
However, to achieve higher accuracy, the effect of regularization needs
to be corrected by using analytic expressions (e.g.,~asymptotic
analysis) for the integrand \cite{beale}.
Finally, there exist special high-order quadrature schemes for domains
with corners, either via reparametrization \cite{Kress91,LIE},
panel-wise geometric refinement \cite{helsingtut}, or by custom
generalized Gaussian quadratures 
\cite{bremer10,bremer2d}.

We now turn to \ns integrals (evaluation close to the surface),
which has traditionally been handled as a distinct task
\cite{helsing_close,lsc2d,beale,hackbusch1994,khayat2005,tlupova2013nearly,helsingtut}.
Beale and coauthors~\cite{yingbeale,beale2015,tlupova2013nearly} use
regularization methods to remove the singularity of the integral.
To correct the error introduced by the regularization, they perform
asymptotic analysis and find correction expressions.
Some authors used singularity cancellation (e.g., using local
polar coordinates) in evaluating \ns integrals
\cite{hackbusch1994,khayat2005}.
Interpolation along carefully-chosen lines connecting distant points
(where a smooth quadrature is accurate) to an on-surface point has
also been successful \cite{ying06,quaife}.

Recently, unified approaches to on-surface and close evaluation have
been proposed, the first  being the \twod Laplace high-order global
and panel-based quadratures of Helsing~and~Ojala~\cite{helsing_close}.
This approach has been extended to \ns Stokes single- and \dl kernels with
global \cite{lsc2d} and panel-based \cite{ojalastokes} quadrature.
 The use of {\em local expansions}\emdash/analytic separation of
variables to the \pde solutions analogous to a Taylor series in the
complex plane\emdash/for the evaluation of integrals near the boundary
was introduced in \cite{ce}.

In this scheme, a refined smooth quadrature is needed to accurately
evaluate the expansion coefficients via the addition theorem.
It was observed that the expansion can also be used to evaluate at
target points on the boundary of the domain, if certain conditions are
satisfied \cite{QBX2}; this was used to construct a unified
quadrature scheme\emdash/\qbxlong (\qbx)\emdash/for near and
on-surface evaluation of integrals \cite{QBX}.
Racch~\cite{Rachh2016} recently showed how to efficiently combine \qbx
evaluations with the fast multipole method.

However, powerful as they are, \qbx schemes require both a local
expansion and addition theorem particular to each \pde, which would be
algebraically tedious especially for vector-valued {\pde}s such as
Stokes and elastostatics.
This motivates the need for a scheme that can handle multiple {\pde}s
without code changes. The present work fills this gap.

\para{Overview and model problems}
As with \qbx, we construct an approximate representation for \pde
solutions in a small region abutting the boundary, then use it for
near and on-surface evaluations.
However, in contrast to \qbx, our representation is an {\em equivalent
  density} on a closed curve enclosing this region; when discretized,
this gives a ring of ``proxy'' point sources (also known as the {\em
  method of fundamental solutions} \cite{Bo85}).
Matching is done at a second smaller ring of ``check'' points where a
refined smooth quadrature is accurate, thus the only dependence on the
\pde is via point-to-point kernel evaluations\emdash/the method is
\ki, and essentially \pde-independent.

We focus on Dirichlet boundary-value problems
\begin{align}
  \op{L} \u & = 0 \quad\text{in } \Omega~,\label{eqn:generic-pde}\\
  \u &= \scalar{f} \quad\text{on } \Gamma~,\label{eqn:generic-bc}
\end{align}
where $\Omega$ is a simply-connected interior domain with smooth boundary
$\Gamma$, for the following partial differential operators:
\begin{equation}
  \op{L} \u =
  \begin{cases}
    \Delta \u                               & \text{Laplace},                                      \\
    (\Delta-\lambda^2)\u                    & \text{Yukawa},                                       \\
    ( \Delta+\omega^2)\u                    & \text{Helmholtz} \quad (\Im\omega\ge 0),             \\
    \Delta \u-\Grad \scalar{p}              & \text{Stokes} \quad (\text{subject to } \Div \u =0), \\
    \Delta \u+\frac{1}{1-2\nu}\Grad \Div \u & \text{Elastostatic}.
  \end{cases} \label{eqn:pde-type}
\end{equation}
To obtain well-conditioned formulations of the problem, we represent the
solution of \cref{eqn:generic-pde,eqn:generic-bc,eqn:pde-type} for
$\vx\in\Omega$  by the \dl potentials
\begin{equation}
  \label{eqn:generic-bi-sol}
  \u(\vx) = \conv{D}[\scalar{\phi}](\vx) \defeq \int_{\Gamma}
  \parderiv{\Phi(\vx,\vy)}{\vn_\vy} \scalar{\phi}(\vy)\, \d s(\vy)~,
\end{equation}
where $\Phi$ is the fundamental solution for the operator $\op{L}$,
and $\phi$ is an unknown density.
The fundamental solutions for the operators listed in
\cref{eqn:pde-type} are given in \cref{apx:kernels}.
A standard step (
see, e.g., \cite{HW}) is now to substitute \cref{eqn:generic-bi-sol}
into the boundary condition and use the jump relation for the potential to obtain
the second-kind integral equation
\begin{alignln}
  \label{eqn:generic-den}
  -\frac{1}{2} \scalar{\phi}(\vx) + (D\scalar{\phi})(\vx) =
  \scalar{f}(\vx), \quad\text{for } \vx \in \Gamma~,
\end{alignln}
where $D$ is the restriction of $\conv{D}$ to the curve. Here, the
integral implicit in the integral operator $D$ must be taken in the
principal value sense.

\para{Discretization and overall approach}
In general, a smooth quadrature is a set of nodes
$\xx_i \in \Gamma$ with associated weights $w_i$, such that
\begin{alignln}
  \int_{\Gamma} \scalar{f} \d s \approx \sum_{i=1}^N w_i
  \scalar{f}(\xx_i)~,
  \label{eqn:smooth-quad}
\end{alignln}
holds to high accuracy for smooth functions on
$\Gamma$\emdash/including the density $\phi$.
In this work, we use $q$-node \gl quadrature scheme on panels, and
for convergence tests, we increase the number of panels while holding $q$ fixed.
Upon discretization, \cref{eqn:generic-den} will be approximated by
the linear system
\begin{equation}
\sum_{j=1}^N A_{ij} \phi_j \; = \; f(\vx_i), \qquad i=1,\ldots,N~,
\label{A}
\end{equation}
whose solution $\vector{\phi} = \{\phi_j\}_{j=1}^N$
approximates the density values at the collocation points.
In practice, for large problems, the matrix $A$ is not constructed
explicitly, but instead the matrix-vector product $A\vector{\phi}$ is
evaluated using the fast multipole method.
We  test the \qbkix scheme both for applying matrix $A$ (i.e.,
on-surface evaluation) and  evaluating the solution at arbitrary
points, near-evaluation in particular.

The system matrix elements are computed using the
\nystrom method \cite[Ch.~12]{LIE}.
If the operator $D$ is  smooth on
$\Gamma \times \Gamma$, we use a smooth \nystrom formula; e.g., for
Laplace,
\begin{equation}
A_{ij} = \left\{\begin{array}{ll}   \parderiv{\Phi(\xx_i,\xx_j)}{\vn_{\xx_j}} w_j, & i\neq j,\\
-\frac{1}{2} -\frac{\kappa(\xx_j)}{4\pi}w_j,& i=j,\end{array}\right.
\label{Anyst}
\end{equation}
where $\kappa(\xx)$ is the curvature at $\xx\in\Gamma$. This
discretization achieves super-algebraic convergence.
However, for Yukawa and Helmholtz in \twod, and all \threed elliptic
kernels, singular quadrature is needed.


In contrast to established approaches using specialized singular
quadratures, we follow the idea underlying the \qbx method:
\emph{applying $A$ to a vector $\bm\phi$ is equivalent to evaluating
  the interior limit of the double-layer potential due to a smooth
  density interpolated from $\bm\phi$}.
This observation leads to the \qbkix idea:
use a fast algorithm combined with the smooth quadrature scheme, \cref{eqn:smooth-quad},
for point evaluation \emph{away} from the
surface\emdash/at points we refer to as \emph{check points}\emdash/
and interpolate from these points to the on surface point, to compute
$A\vector{\phi}$ for the Krylov iteration.
As this interpolation can be done using points on  one or both sides of the surface, in
\cref{ssc:spectrum} we  compare ``one-sided'' and ``two-sided''
variants of \qbkix with respect to their spectra and iterative
convergence rates.

Although we are focusing on interior Dirichlet
tests and \nystrom-style sampled representation of the density
in this work, \qbkix is applicable for Neumann or other boundary
conditions,
and Galerkin and other discretization types.
Moreover, while the approach presented in this paper is restricted to \twod,
there is no fundamental obstacle to an extension to \threed.

The rest of the paper is structured as follows.
In \cref{sec:formulation} we present the \qbkix algorithm for
integration. We present an error analysis in \cref{sec:error}.
In  \cref{sec:results}, and report the results of numerical
experiments quantifying the accuracy of the method for a number of
representative problems.

\section{Algorithms\label{sec:formulation}}
Given a closed curve $\Gamma\subset \RR^2$ with interior $\Omega$, and
Dirichlet data $f$ on $\Gamma$, our goal is to numerically solve the
integral equation~\eqref{eqn:generic-den} for density
and evaluate the solution of the underlying \pde
at an arbitrary target point $\xx\in\overline{\Omega}$.
We assume that $\Gamma$ is parametrized by a $2\pi$-periodic
piecewise-smooth function $\vX(t)$, so that the arc length element is
$\d s = |\vX'(t)|\d t$, $|\vX'(t)|$ is bounded from below,  and
that $\vX(t)$ and the data function $f(t)$ may be evaluated at any
$t\in [0,2\pi)$. The boundary is subdivided into \emph{panels},
 which can be of different lengths, on which
the native quadrature rule is defined (we use \gl quadrature), at $q$
nodes $\xx_j$ per panel.
We assume that the density is available as a vector of samples
$\phi(\xx_j)$ at the quadrature nodes.

\subsection{Single-point evaluation}
We describe our method in the simplest form for computing the solution
accurately at a given point $\vx$. We assume that there is a single
point on $\Gamma$ closest to $\vx$, on a panel of length $L$.
We assume that  at a distance $2\delta$
along the normal to the panel at any point,  the native quadrature
meets the target accuracy of evaluation,  so the
distance from $\vx$ to the surface is less than $2\delta$.
We discuss how $\delta$ is chosen and how to ensure that this
condition holds after the algorithm formulation.

The local geometric configuration of various types of points we are
using in our algorithm is shown in Figure~\ref{fig:stages}.  The setup
shown in the image is for computing the potential accurately for any
point $\vx$ inside a disk $\evald$ of radius $\delta$ centered at $\cen$, touching the
surface at a point $\vx_0$ on a panel of length $L$.

The points we use in the algorithm are placed on two concentric
circles with the same center as the evaluation disk $\evald$.
The \emph{proxy  points} on a circle $\proxyc$ of a radius $R>\delta$,
where we compute \emph{equivalent density} values, are used to approximate
the solution inside $\evald$.
The \emph{check points} $\vz_i$ are  on a circle $\checkc$ of a radius $r_c
< \delta$.  At these points, we evaluate the solution accurately by
using a smooth quadrature on panels refined by a factor $\beta$.
The check points are used to compute the equivalent density values at
the proxy points as described below.
\begin{figure}[!bt]
  \centering
  \setlength\figureheight{2.5in}
  \setlength\figurewidth{1.5\figureheight}
  \includepgf{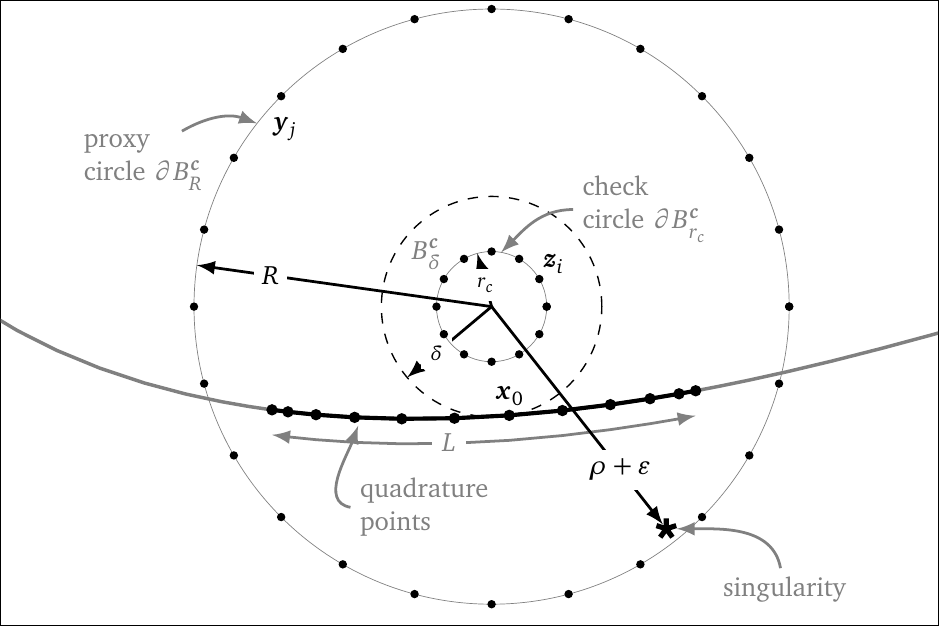}
  \mcaption{fig:schematic}{Schematic of a \ki expansion}{Geometry of
    \qbkix, with proxy and check circles  centered at $\cen$ near a panel of
    length $L$ of the boundary $\Gamma$ discretized with $q$ \gl
    sample points. The evaluation domain $\evald$ is a disc centered at
    $\cen$ of radius $\remax$ (dashed circle abutting the boundary at $\vx_0$).
    The points $\vz_i$ are the check points on the
    circle $\checkc$ of radius $r_c$, and  $\vy_j$ are the proxy points on
    the circle $\proxyc$ of radius $\rp$.
    For error analysis, the singularities of the exact solution are
    assumed to be at a distance farther than $\rho$ from $\cen$.
    Note that, for clarity, the relative sizes of circles and distances
    between samples are different from the ones actually used.}
\end{figure}

The algorithm depends on a number of parameters; these parameters need
to be chosen appropriately to achieve an overall target accuracy. Specific choices
are discussed in the next section. The key steps in the algorithm are
\begin{enumerate}[(1)]
\item \lhb{Set-up of proxy and check points.}
  We  choose a center $\cc \in \Omega$ at a
  distance $\delta$ from $\Gamma$, such that $\vx$ is no further from
  $\cc$ than $\delta$.
    E.g., for $\vx\in\Gamma$, we set $\cc = \vx - \delta \vn$, where $\vn$ is
    the outward normal.
  $\np$  proxy points $\vy_j$ are arranged equally on the circle
  of radius $R$ with center $\cc$, where $R>\delta$ is of order $L$.
  Similarly $\nc$ check points $\zz_i$ are arranged on the
  concentric circle of radius $r_c<\delta$ (\cref{fig:schematic}).
  %
\item \lhb{Upsampling the density.}
  Each panel is split into $\beta$  panels corresponding to equal
  ranges of $t$,   to give a set of $\beta N$  fine-scale nodes $\tilde{\xx}_l$
  with weights $\tilde{w}_l$. The global factor $\beta$ is chosen so that
  the solution can be evaluated accurately at the check points, i.e.,
  at a distance $\remax -r_c$ from the surface.
  The density is interpolated from its original samples $\phi(\xx_j)$
  on each panel, using \ordinal{q} order Lagrange interpolation to the
  fine-scale nodes,   to give the refined vector of samples
  $\tilde{\phi}_l$, $l=1,\ldots, \beta N$.
  %
\item \lhb{Direct upsampled evaluation at check points.}
  The integral is evaluated at each check point $\zz_i$ using the
  fine-scale boundary native quadrature:
  \begin{equation}
    \label{upsampled}
    \un(\zz_i) = \sum_{l=1}^{\beta N} \frac{\partial
      \Phi(\zz_i,\xx_l)} {\partial \vn_{\xx_l}} \tilde{\phi}_l
    \tilde{w}_l~.
  \end{equation}
  Denote by $\unv \defeq \{\un(\zz_i)\}_{i=1}^{\nc}$ the
  column vector of these values at the check points.
  %
\item \lhb{Solving for the equivalent density values.}
  Next, we construct an $\nc\times\np$ matrix $Q$ with elements
  \begin{equation}
    Q_{ij} = \Phi(\zz_i,\yy_j)~.
    \label{Q}
  \end{equation}
  Applying $Q$ to a vector of density values at proxy points computes
  a periodic trapezoidal rule  approximation to the single-layer
  potential corresponding to this density evaluated at check points.
  Then we solve a small, dense, and ill-conditioned linear system
  \begin{equation}
    Q \bal = \unv~,
    \label{linsys}
  \end{equation}
  in the least-squares sense,  to get the set of proxy density values
  $\bal := \{\alpha_j\}_{j=1}^{\np}$.
  The ill-conditioning arises from the exponential decay of singular
  values in the single-layer operator between concentric circles (see
  \cref{fig:sing-vals}).
  Despite this, if \cref{linsys} is solved in a backward-stable
  manner, a high-accuracy result is obtained (cf. \cite{mfs}, we
  explain the details below for completeness).
\item \lhb{Evaluation of the proxy sources at the target.}
  Finally, the equivalent density is evaluated at the target $\vx$,
  \begin{equation}
    \label{eqn:localexp}
    \uq(\vx) = \sum_{j=1}^{\np} \scalard{\alpha}_j \fundsol(\vx,\vy_j)~,
  \end{equation}
  We may view this as an approximation for the true solution $u$ in the basis
  of fundamental solutions centered at the proxy points, that
  holds to high accuracy in the disk $\evald$. 
\end{enumerate}

\Cref{fig:stages} illustrates the stages of \qbkix evaluation
for a set of target points lying in a single disk $\evald$.
The final evaluation of \cref{eqn:localexp} over the disc of target points
has around 12 digits of accuracy.

\begin{figure}[!bt]
  \centering
  \small
  \setlength\figureheight{1.4in}
  \setlength\figurewidth{1.7in}
  \includepgf{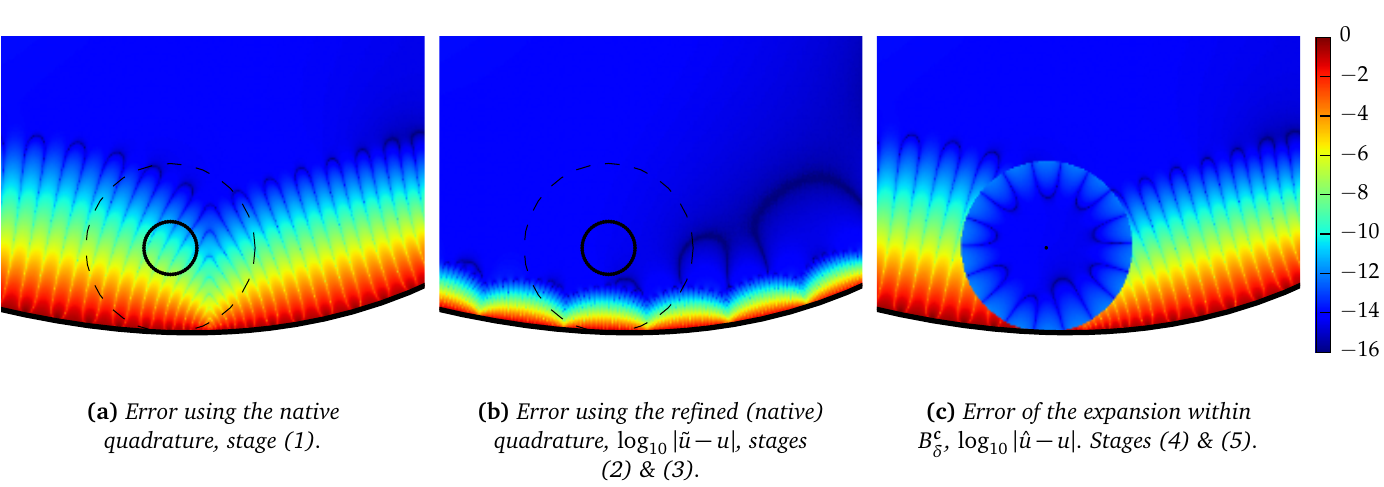}        
  \mcaption{fig:stages}{Stages of \qbkix construction}{ The stages
    given in \cref{sec:formulation} are illustrated using plots
    of the $\log_{10}$ of the evaluation error near the boundary, for
    the double-layer density $\phi\equiv 1$ for Laplace's
    equation. The evaluation disc $\evald$ (dashed circle), check circle $\checkc$
    (solid circle) are shown, and proxy points are not shown.
  }
\end{figure}

\para{Handling the ill-conditioned linear solves}
The ill-conditioned system \cref{linsys} is solved by applying a regularized
pseudo-inverse, as follows.
Let $\apinv$ be the desired relative accuracy for inversion; typically we set
$\apinv=\sci{-14}$.
Then, taking the singular value decomposition (\abbrevformat{SVD})
\cite{nla} $Q = U \Sigma V^*$ with $\Sigma = \mbox{diag}\{\sigma_j\}$
being the diagonal matrix of singular values, we write $\Sigma^\dagger
\defeq \mbox{diag}\{\sigma^\dagger_j\}$ where
\begin{equation}
  \sigma^\dagger_j = \left\{ \begin{array}{ll} \sigma_j^{-1}, & \quad
    \sigma_j>\apinv \sigma_1,\\ 0, & \quad
    \mbox{otherwise}. \end{array} \right.\label{eq:pinv-sing-val}
\end{equation}
Then we use the solution
\begin{equation}
  \bal := V (\Sigma^\dagger U^* \vu)~.
  \label{pinv}
\end{equation}
Note that the matrices $U^*$ and $V$ must be applied in two
separate steps (as indicated by the parenthesis) for backward
stability \cite{nla}, since a matrix-vector multiply with the single
pseudo-inverse matrix $Q^\dagger := V \Sigma^\dagger U^*$ is unstable
due to round-off error caused by its large entries.
If $k$ is the number of singular values greater than $\apinv$, i.e., the
numerical $\apinv$-rank of the matrix $Q$, the factors $V$ and $U^*$
have sizes  $\np \times k$ and  $k \times \nc$ respectively.

\para{Parameter summary} The algorithm described above uses a number of
parameters, which we summarize here.

The following parameters are defined globally:
\begin{itemize}
\item The quadrature order $q$, which determines the number of samples
  per panel, and both far-field evaluation accuracy and, together
  with $\beta$, the accuracy of evaluation at check points. This
  parameter is selected arbitrarily based on the desired overall
  accuracy. We use $q = 16$, which is sufficient for full double
  precision of integration in the far field.

\item The panel refinement factor $\beta$ which needs to be chosen to
  maintain desired accuracy for check point evaluation.

\item The numbers of proxy points $\np$ and check points $\nc$;
  the former determines how accurate the approximation inside $\evald$ can
  be and the latter is chosen to have enough sampling.
\end{itemize}

Three additional parameters, the accurate evaluation distance $\delta$, the
proxy point circle radius $\rp$ and the check point circle radius
$r_c$, are panel-dependent, and are chosen with respect to panel size
$L$. A careful choice of all of these, as fractions of $L$, is needed
to achieve a target error without requiring excessive refinement.
We discuss the choice of these parameters in Section~\ref{sec:error}.

\para{Defining panels}
In our experiments, we consider two ways of defining panels. The first
approach is primarily needed to understand the convergence of the
method with respect to the number of panels, i.e., for a given number
of panels, we determine the error.  In this case, we simply
partition the parametric domain of $\vX(t)$ into $M$ equal-sized
intervals, with one panel corresponding to each interval.  We assume
the parametrization to be sufficiently close to an arclength
parametrization, so that the panel length has little variation,
and choose $M$ to be fine enough so that the geometric condition on
the check points is satisfied.

In a more practical scenario, when a target error is specified,
we need to determine panel sizes adaptively.
The key requirement that needs to be satisfied by panels is that
the accuracy of check-point evaluation at stage 2 matches the target
accuracy in the far field (i.e., points farther than $2\remax$ from the boundary).
%
%
%
The adaptive refinement starts with one  panel covering the entire
boundary, then recursively splitting panels into two equal pieces in
parameter $t$, until all panels are deemed \emph{admissible} or their length is less than a set tolerance $\acc$.

A panel is admissible if
\begin{inparaenum}[(i)]
\item the interpolation of $\vX(t)$ and $f(t)$ from a $q$-node panel
  at the collocation points of the two $q$-nodes \gl panels (obtained
  by splitting the coarse panel to two pieces) matches the direct
  evaluation of $\vX$ and $f$ on the finer nodes, to a maximum
  absolute tolerance $\atol$, which we choose as $\sci{-11}$ unless
  stated otherwise;
\item it is no more than twice the parameter length of that of its
  neighbors;
\item the length of the panel does not exceed a given fraction of the minimal radius of curvature at a point of the panel, \emph{or} is less than a minimal length proportional to the target error; and
\item any check point corresponding to a point $\xx$ is not closer than
  $\remax-\rc$ to any point on the surface.
\end{inparaenum}

The second criterion ensures that the panels are the leaves of a
\emph{balanced} binary tree, which is needed for accurate evaluation
of integrals at the check points. For domains with sharp corners, the
forth and second conditions imply dyadic refinement of panel length
bounded below by panel minimum length $\acc_l$.

In both cases, the result is a set of $N$ nodes $\xx_j = \vX(t_j)$,
where $t_j$ are the parameter values of the nodes, with weights $w_j =
|\vX'(t_j)| w'_j$ where $w'_j$ are the \gl weights scaled by the panel
parametric lengths.
This native quadrature approximates the
boundary condition $f$ with target accuracy $\atol$.
It follows from \cref{eqn:generic-den} that this also holds for the density $\phi$,
as $\phi$ to be no less smooth than $f$ and $\vX$.


\subsection{On-surface evaluation for iterative solution of the linear system}
\label{s:sided}

As discussed in the introduction, one context where singular
quadratures are needed is for applying $A$, the matrix discretization
of the operator $(-\half I +D)$, to the current density vector
$\bm\phi$ during the iterative solution of \cref{A}.
This matrix-vector multiplication is equivalent to evaluation of the
interior limit of the double-layer potential at the nodes due to the
smooth interpolant of the density vector.
As with \qbx \cite[Sec.~3.5]{QBX}, one may exploit this in two
different ways.
\begin{itemize}
\item One-sided \qbkix: as stated above, we use the interior limit of
  the potential at the nodes for $A\bm\phi$.
\item Two-sided \qbkix: we average the interior and exterior limits of
  the potential at the nodes, which, by canceling the jump relation
  terms, applies a matrix approximation to the operator $D$. We then
  explicitly add $-\half \bm\phi$ to the answer.
\end{itemize}
Although mathematically equivalent, these two variants smooth
high-frequency components in the density differently: one-sided \qbkix
tends to dampen these components, leading to an accumulation of
eigenvalues of $A$ around zero. This has a negative impact on
convergence.
In contrast, for two-sided \qbkix, since the approximation of $D$ tends to damp high-frequency components,
the explicit inclusion of $-\half I$ ensures that these components end up
being multiplied by a number very close to $-\half$, which leads to
better clustering of the spectrum and improved convergence rates.
We present a numerical comparison of these two alternatives in \cref{ssc:spectrum}.

\subsection{Efficiency considerations and computational
  complexity\label{s:complex}}

Given a set of evaluation points $\xx$, the brute-force approach is to
run the algorithm described above, including construction of check and
proxy points, for each sample point separately.
This is highly inefficient, and the following obvious optimizations
can be applied:
\begin{itemize}
\item  The upsampled density on the fine-scale nodes need be computed
  only once, and each expansion center may be chosen to cover several
  targets; this requires increasing evaluation disk radius $\delta$,
  adjusting other parameters accordingly.

\item The \abbrevformat{SVD} of matrices $Q$ may be precomputed.
  For  translation- and scale-invariant kernels,  (i.e., all kernels
  we consider except Yukawa and Helmholtz) these matrices do not depend on the choice
of the center and circle radii, as long as the ratio $\rp/\rc$ is fixed.
\item One may use the kernel-independent FMM  method 
for evaluation of the solution at the check points for all target points at once.
\end{itemize}

We consider the complexity of using  \qbkix for the task of
on-surface evaluation at all  boundary nodes $\vx\in\Gamma$.
For a boundary with $M$ panels and $q$-node \gl quadrature on each,
there are $N=Mq$ nodes in total.
We use a conservative assumption that a distinct set of check and proxy points is used for
each of the targets.
Then, using \abbrevformat{KIFMM}, the evaluation of the boundary
integral from the $\beta$-refined boundary to the check points is
\O{(\beta + \nc)N}.
We assume that the factorization of the pseudo-inverse for computing
the equivalent densities $\bal$ is precomputed. The cost of applying the factors
$V$ and $U^*$, of sizes $\np \times k$  and $k \times \nc$,
for targets point is \O{k(\nc+\np)N}.
The cost of evaluation of the approximation from proxy density values at target points
is \O{N\np}.

We conclude that the overall cost is \O{(\beta + \nc + k\nc + k\np + \np)N},
which for typical choices $\beta=4$ and $\nc=2\np$ reduces to \O{k\np N}.
We see that the scheme is linear in $N$, but with a prefactor of order
$k^2$ (since, as discussed in the next section, $\np$ is of order $k$).
The two-sided variant involves another overall factor of 2.

If the same check and proxy points are used for a number of targets,
an additional, potentially very large, constant-factor speedup can be obtained.
The speedup factor is proportional to the average number of
targets handled by each set of check and proxy points.


\section{Error analysis and parameter choices\label{sec:error}}
In this section, we present theoretical results, focusing on the cases
of scalar $\u$ governed by the Laplace equation $\Delta \u =
0$\emdash/or by the Helmholtz equation $(\Delta + \omega^2)\u = 0$ for
real $\omega$.  We expect similar results for other elliptic {\pde}s
in \cref{eqn:pde-type}.

We split \qbkix into two stages:
\begin{inparaenum}[(i)]
\item evaluation of $\u$ on the check points using a refined native
quadrature, with the associated error $\ech$;
\item solution of a small linear system to determine the
equivalent density values $\bal$ at the proxy points that best represent $\u$ at the check
points. This is followed by evaluating the approximation of $u$ at target
points using these density values.
\end{inparaenum}

At the first stage, the error $\ech$  is effectively the smooth quadrature error
of the refined panels.  The primary focus of our analysis is on the
second stage.  We analyze the error behavior in the idealized
situation of exact arithmetic and infinitely many check points,
obtaining the dependence of the second-stage error $\err$ on $\delta$,
$\rp$, $\rho$, and $n_p$.
We then describe a heuristic model for the effects of finite-precision
computations, which adds an extra term to $\err$, depending on $\ech$,
$\delta$,  $r_c$, and $k$.

We use the overall error model, along with experiments, to provide
a choice of the various parameters in the scheme resulting in the
on- and near-surface evaluation errors of the same magnitude as the
far-field integration errors.


\subsection{Error at check points\label{ssc:err-check}}
Recall that evaluation of $\u$ on the check points is done by
approximating the exact integral \cref{eqn:generic-bi-sol} by
\cref{upsampled} using $q$-node \gl quadrature on panels (subdivided
by factor $\beta$).
For a flat panel, the error $\ech$ in this evaluation is bounded by
standard quadrature estimates giving a term of the form
$C_q (L/(4 \beta d))^{2q} \|\phi\|_{C^{2q}}$ where  $d =
\remax-\rc$ is the closest distance of check points to the panel, and
$\phi$ denotes the density for which we evaluate the integrals.
Our adaptive refinement procedure ensures that the formula still holds,
as the radius of curvature of the panel is larger than its length, and
hence larger than $\remax$.

This estimate has the form of the second term in
\cite[Theorem~1]{QBX}, and for convergence as the panel
length $L$ going to zero, it requires $L/d$ to converge to zero as well.
Instead of following this route, we  fix the ratio $L/d$ to a
constant, by choosing $\remax$ and $\rc$ as fractions of $L$.
If $L/(4 \beta d)$ is sufficiently small, a high-order quadrature
for sufficiently large $q$ allows us to compute the integrals
with any desired precision. For instance, when $q=16$,
it is sufficient to use $L/(4\beta d) = 1/2$, to obtain an error on
the order of $\sci{-10}$ at distance $d$ from the panel.



\subsection{Error of the proxy point representation in exact arithmetic\label{ssc:err-exact}}
Next, we analyze the dependence of the error (computed in exact arithmetic) of
the second stage of \qbkix on the number of proxy points $n_p$,  the proxy circle radius $\rp$,
and the distance $r$ from the center $\cen$ to the evaluation point.
The distance  $\re$  could be either smaller than $\remax$ if targets
are away from the surface, equal to $\remax$ if $\evald$ touches the
surface at a single point,  or exceed $\remax$ if there
are several on-surface targets in $\evald$; we focus our attention to the case where $\re \le \remax$.

Let $\uq$ be given by the proxy representation, \cref{eqn:localexp}, with equivalent density values $\alpha_j$ at proxy points $\vy_j$,
$j = 1,\ldots, \np$. We consider evaluation of the approximation $\uq$ in $\disc{B}{\re}{\cen}$,
the disc of radius $\re$ centered at $\cen$, given correct values
for $\u$ at a very large number of check points $\nc$, so that we can
replace the discrete least-squares problem we solve with a continuous
one.


Let the equivalent densities $\scalard{\alpha}_j$ be chosen to minimize the $L^2$ error on the
check circle, i.e.,
\begin{alignln}
  \vectord{\alpha} = \arg \min_{\vectord{\alpha}\in\mathbb{C}^{\np}}
  \|\uq - \u\|_{L_2(\checkc)}~.
\label{npinf}
\end{alignln}
By convergence of the periodic trapezoidal quadrature on the check
points, this corresponds to the $\nc\to\infty$ limit of the \qbkix
scheme.
Let
\begin{equation}
 \err(\re) := \sup_{\vx\in\overline{\Omega}\,\cap\,\disc{B}{\re}{\cen}} |
 \uq(\vx)-\u(\vx)|~, \label{sup}
\end{equation}
be the upper bound on the pointwise error in the part of the disc
lying inside the closure of the domain.
We have the following bounds on $\err$ when $\u$ is sufficiently
regular, meaning that any singularities in the continuation of $\u$ is
further than some distance $\rho>\remax$ from the center of the expansion
$\cen$.

\begin{theorem}
  \label{thm:expconv}
  Let $\u$ be continuable as a regular solution to the Laplace or
  Helmholtz equation in the closed disc of radius $\rho$ centered at
  $\cen$.
  Let $\rp > \remax$ in the Laplace case.
  Let the \qbkix equivalent density values at proxy points be
  solved in exact arithmetic in the
  least-squares sense on the check circle as in \cref{npinf}, and let
  $\err$ be defined by \cref{sup} where $\uq$ is the expansion in
  \cref{eqn:localexp}.
  Then, in a disc of radius $\re$
  \begin{alignln}
    \err(\re) \; \le \left\{\begin{array}{ll}
    C \bigl( \frac{\re}{\rho} \bigr)^{\np/2}, & \rho \re < \rp^2~, \\
    C \sqrt{\np} \bigl( \frac{\re}{\rp} \bigr)^{\np},    & \rho \re = \rp^2~, \\
    C \bigl( \frac{\re}{\rp} \bigr)^{\np},    & \rho \re > \rp^2~,
    \end{array}\right.
    \label{eqn:expconv}
  \end{alignln}
  where in each case, $C$ indicates a constant that may depend on $\u$
  (and $\omega$ in the Helmholtz case), $\cen$, $\re$, and $\rp$ but
  not on $\np$.
\end{theorem}
\begin{proof}
  Following the technique of Barnett and
  Betcke~\cite[Theorem~3]{mfs}, we only need to show that there exists
  some choice of density values $\vectord{\alpha}$ for which the
  estimate holds; the least-squares solution cannot be worse
  than this.
  We choose density values $\vectord{\alpha}$ to cancel the Fourier
  coefficients with frequency $|n|<\np/2$ of the pointwise error $\uq -
  \u$ on the check circle.

  By uniqueness of the local expansion for the regular \pde
  solution (in polar coordinates, $\sum_{n\ge 0}
  a_n r^n e^{in\theta}$ for Laplace or $\sum_{n\in\mathbb{Z}} a_n
  J_n(\omega r) e^{in\theta}$ for Helmholtz) this choice of density
  values  also cancels the same Fourier coefficients on any circle centered at $\cen$ with radius less than $\rp$.
  Applying \cite[Theorem~3]{mfs} for the Helmholtz case, the
  $L^2$-norm of the error on the circle of radius $\re$ obeys a bound
  of the form \cref{eqn:expconv}.
  Barnett and Betcke~\cite[Section~2.1]{mfs} produce the Laplace case as a limit of the
  Helmholtz case; however, one also needs the result that the constant
  single-layer density generates the constant potential $\log \rp/\rc$,
  which excludes $\rp=\rc$ because it can only produce zero-mean data on
  the circle.

  Finally, we need to show that the sup norm of the error on the
  circle of radius $\re$ is bounded by the $L^2$-norm; this holds
  since the error $\uq - \u$ is a regular \pde solution in a disc with
  radius strictly larger than $r$, namely $\disc{B}{\min(\rp,\rho)}{\cen}$.
  Thus, its Fourier coefficients on the $r$-circle decay exponentially
  in $|n|$,
  and are thus summable with a bound controlled by the $L^2$ norm.
  In the case where $\disc{B}{\re}{\cen}$ lies partially outside $\Omega$, one
  may continue $\u$ as a regular \pde solution in the disc and apply
  the above.
  \qed
\end{proof}

\begin{remark}
  The above derivation relies on analysis from the  literature on the
  method of fundamental solutions (\abbrevformat{MFS}).
  The original result for the Laplace equation is due to
  Katsurada~\cite[Theorem~2.2]{Ka89}, which considers the case
  $n_c=n_p$ and restricted to $\re=\rc$.
  We  extend this result to include {\em extrapolation} from the check
  radius $\rc$ out to larger radii $\re$.

  Remarkably, $\rc$ does not appear in \cref{eqn:expconv}, because in
  exact arithmetic it does not matter at what radius the Fourier
  coefficients are matched.
  In the next section we will see that in practice rounding error
  strongly affects the choice of $\rc$ since the extrapolation is ill-conditioned.
  \label{r:rcindep}
 \end{remark}

A surprising aspect of \cref{thm:expconv}
is that $\u$ may have singularities {\em closer} to the
 center than the proxy radius $\rp$ and yet exponential convergence
 still holds; this is closely related to the Runge approximation
 theorem.

\begin{remark}
  The two regimes in \cref{eqn:expconv} may be interpreted as follows:
  \begin{itemize}[\quad$\bullet$]
  \item $\re < \frac{\rp^2}{\rho}$: the solution $\u$ is relatively
    rough (has a nearby singularity), and error is controlled by
    the decay of the local expansion coefficients $a_n$ of $\u$ for
    orders beyond $n_p/2$.
  \item $\re > \frac{\rp^2}{\rho}$: the solution $\u$ is smooth, and
    error is controlled instead by aliasing (in Fourier coefficient
    space) due to the discreteness of the proxy point representation
    on the proxy circle.
  \end{itemize}
  We observe in numerical experiments that when the boundary is
  adaptively refined based on the boundary data as in
  \cref{sec:formulation}, $L\approx\rho$ and the expansion centers that
  dominate the error in a domain are typically those that are near to
  a singularity of the solution.
  Such centers are typically in the rough regime.
\end{remark}

Note that the boundary $\Gamma$ may intersect the closed disc, and
still $\u$ may be continued as a \pde solution into the closed disc.
This requires the boundary data $f$ or density to be
analytic\emdash/see \cite{ce} for related analysis of \qbx in this
case.

\begin{remark}[Extension of analysis to other kernels]
It is clearly of interest to have a \ki extension of
\cref{thm:expconv} that would apply also to vector {\pde}s such as
Stokes.
Initial attempts suggest this requires significantly more complicated
analysis, since to use the method of the above proof
one needs to be able to write down a proxy coefficient vector
$\bm\alpha$ that produces a single Fourier mode on the check circle
plus exponentially decaying amounts of aliased modes, which is
challenging even in the Stokes case.
We leave this for future work.
\label{r:anal}
\end{remark}

\subsection{Modeling the effect of finite-precision arithmetic\label{ssc:err-float}}
Independence from $\rc$ in \cref{thm:expconv} relies on exact
arithmetic;  since the extrapolation from  $\rc$ to a larger $\re$ is ill-conditioned.
Moreover, due to finite precision, there are possibly fewer than $\np$ functions available to cancel the Fourier coefficients.
As a result, we need to study the effect of rounding error on $\uq - \u$.
Rather than attempting a rigorous analysis, we present a heuristic model
and demonstrate that it agrees well with numerical observations.

We first show that the \ordinal{n} singular value of the matrix $Q$ in \cref{Q}
decays as $\frac{1}{n}(\rc/\rp)^{n/2}$, i.e., marginally faster than exponentially.
In the continuous limit ($\np,\nc\to\infty$), this corresponds to the
decay of the eigenvalues of the single-layer operator with kernel $\fundsol$,
whose eigenfunctions are the Fourier modes, since the operator is
convolutional.
For the Laplace equation, the potential defined in polar coordinates
centered at $\cen$ as
\begin{equation*}
  v(r,\theta) = \begin{cases}(\rp/2n) (r/\rp)^n e^{in\theta}~, & r\le
    \rp~, \\ (\rp/2n) (r/\rp)^{-n}e^{in\theta}~, & \text{otherwise}~,
  \end{cases}
\end{equation*}
solves the \pde everywhere except at $r=\rp$, where the jump in radial
derivative is $e^{in\theta}$.
We conclude that  $v$ is the single-layer potential due to the $\ordinal{n}$
Fourier mode density.
Substituting $r=\rc$, and recalling that the
$\ordinal{n}$ singular value is eigenvalue for the frequency $n/2$, as
the frequencies are in the range $-n/2$ to $n/2$, we conclude that
$ \sigma_n = \frac{1}{n}(\rc/\rp)^{n/2}$.

The above argument also applies for the Stokes case except due to having two
vector components, \ordinal{n} singular value of matrix $Q$ corresponds to the eigenvalue for frequency $n/4$.
The Helmholtz case\emdash/although there are \O{\omega} eigenvalues that do
not decay\emdash/is asymptotically identical to Laplace
\cite[Equation~(14)]{mfs}.
To verify this asymptotic behavior,  in \cref{fig:sing-vals} we show the decay of singular
values for several kernels.
\begin{figure}[!b]
  \centering
  \small
  \setlength\figureheight{1.6in}
  \setlength\figurewidth{2.2in}
  \includepgf{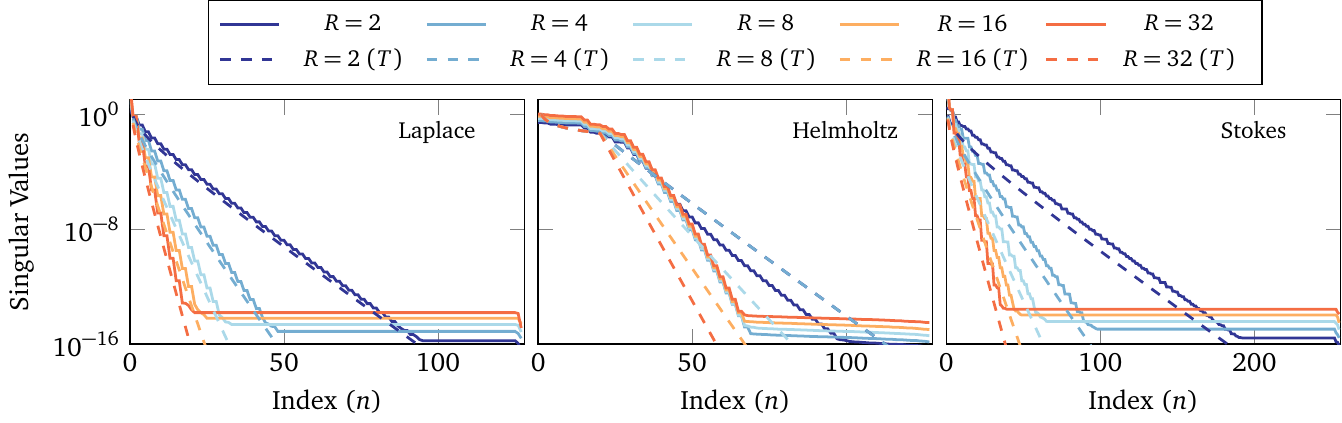}
  \mcaption{fig:sing-vals}{Singular values of proxy to check matrix}{
    The solid lines are the singular values of $Q$ for different $\rp$
    and different \sl kernels, and the dashed lines labeled $(T)$ are
    the theoretical decay: $\frac{1}{n}(\rc/\rp)^{n/2}$ for Laplace or
    Helmholtz, and $\frac{1}{n}(\rc/\rp)^{n/4}$ for Stokes, where $n$
    denotes the index of the singular value.
    Other parameters are $\rc=1$, $\np=128$, $\nc=256$.
    For the Helmholtz problem, the dashed lines show the asymptotic
    bound for the singular values and are not accurate for small
    indices; the interested reader is referred to
    \cite[Eq.~(14)]{mfs}.
    }
\end{figure}

When the pseudoinverse of $Q$ is computed based on \cref{eq:pinv-sing-val},  only
$k$ singular values lying above $\apinv \sigma_1$ are retained.
The corresponding singular vectors approximate the lowest Fourier
modes up to frequency $|n|<k/2$ (in the scalar \pde cases).
Thus, equating up to constants the $\ordinal{k}$ singular value above
to $\apinv$, the ranks of the matrices in the pseudoinverse are
\begin{equation}
k \;\approx\; \min \left(k_m, \, \np\right),\qquad k_m = \,2
\frac{\log(1/\apinv)}{\log(\rp/\rc)} ~,
\label{k}
\end{equation}
and the highest (Nyquist) frequency they can represent is $k/2$.

The values of $\un$ at the check points have error bounded by $\ech$,
so in this model we expect the errors to be
amplified (by considering the local expansion as above) to become
$\ech (\re/\rc)^{k/2}$ at the evaluation radius $\re$.


\subsection{Error bounds and optimal parameter choices\label{ssc:error-bounds}}
Combining the results from \cref{ssc:err-exact,ssc:err-float} for a \ki
expansion, using $\np$ proxy points, the error is bounded by
\begin{alignln}
  \label{eqn:err-bound}
  \err(\re) \; \le \left\{\begin{array}{ll} C \left( \dfrac{\re}{\rho}
  \right)^{k/2}+ C \ech \left(\dfrac{\re}{\rc}\right)^{k/2}, & \rho \re
  < \rp^2~, \\[10pt] C \left( \dfrac{\re}{\rp} \right)^{\np}+ C \ech
  \left(\dfrac{\re}{\rc}\right)^{k/2},& \rho \re > \rp^2~,
  \end{array}\right.
\end{alignln}
where $C$ represents possibly different constants in each case
(omitting the case $\rho r = \rp^2$).

\begin{figure}[!b]
  \centering
  \small
  \setlength\figureheight{2in}
  \setlength\figurewidth{2.2in}
  \hspace{-3em}\includepgf{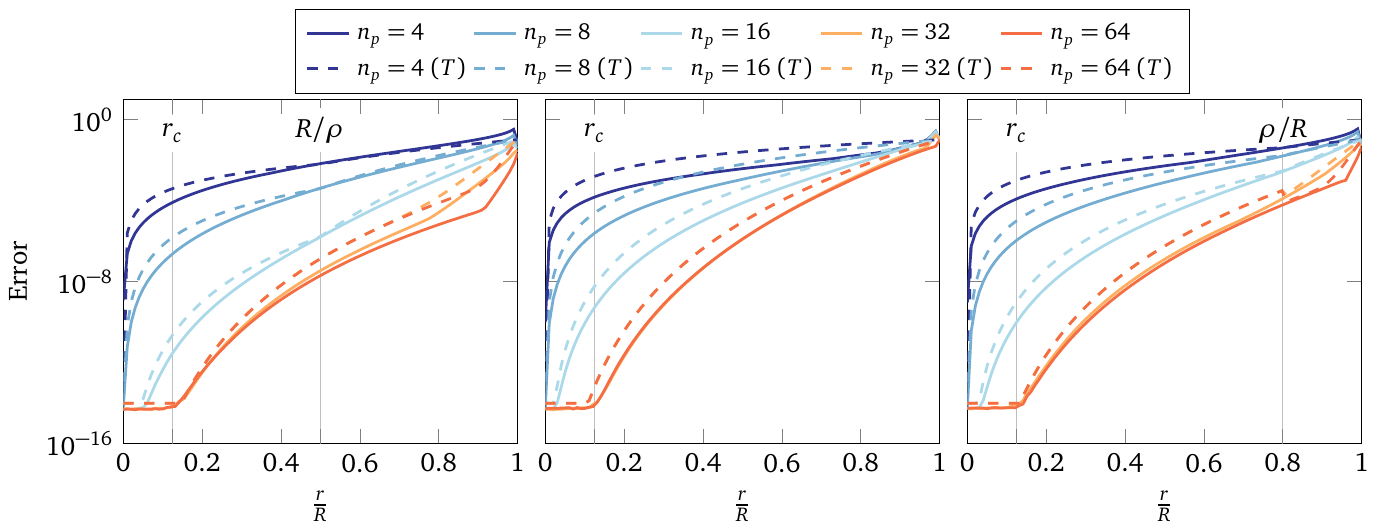}

  \mcaption{fig:err-bound}{Error bounds for Laplace \qbkix with known
    singularity}{Errors $\err$ observed (solid lines) and predicted by
    \cref{eqn:err-bound} (dashed lines) for a single expansion with
    different singularity distances $\rho=2\rp,
    \rp,\text{~and~}0.8\rp$, and different numbers of proxy points
    $\np$.
    The expansion is centered at $\vector{c} = [0,0]$ and the solution
    $\u(\vx) = -\log|\vx-\vx_0|$, $\vx_0 = \rho e^{1\ii/19}$ is a
    harmonic function with a singularity at distance $\rho$.
    Laplace \sl kernel is used for the expansion.
    The error is the maximum error over the $\disc{B}{r}{\vector{c}}$ as
    defined in \cref{sup}.
    The proxy to check radius ratio is $\rp/\rc=8$, the number of
    checks is set to $\nc=2\np$, $\ech=\sci{-14}$, and $k_m \approx 27$
    (given by \cref{k} with $\apinv=\sci{-14}$).
    The constants $C$ in \cref{eqn:err-bound} were chosen to
    qualitatively match the trend lines (all set to $0.1$).

 }
\end{figure}
In \cref{fig:err-bound}, we show how this formula models the error
growth for a single \ki expansion interpolating a Laplace solution in
free space with a known nearest singularity at various distances
$\rho$, for a typical choice of ratio $\rp/\rc = 8$.
The key observation is that, despite its simplicity, our model
\cref{eqn:err-bound} explains well the observed error behavior.
Other salient features of the plots include:
\begin{itemize}
\item As $r$ increases beyond $\rc$, errors grow rapidly  dominated by the second term in the error estimate.
\item The error is mostly controlled by $k$ and increasing $\np$
  beyond $k_m\approx 27$ (defined in \cref{k}) has no tangible effect unless
  $\rho\re>\rp^2$ (i.e., right half of left plot).
\end{itemize}
\Cref{fig:param-reval} instead continuously varies $\rp/\rho$ (the
inverse scaled singularity distance), showing the same effect: a
relatively distant singularity allows high accuracy expansion out to
larger $\re/\rp$.
\begin{figure}[!bt]
  \centering
  \setlength\figureheight{2in}
  \includegraphics[height=\figureheight]{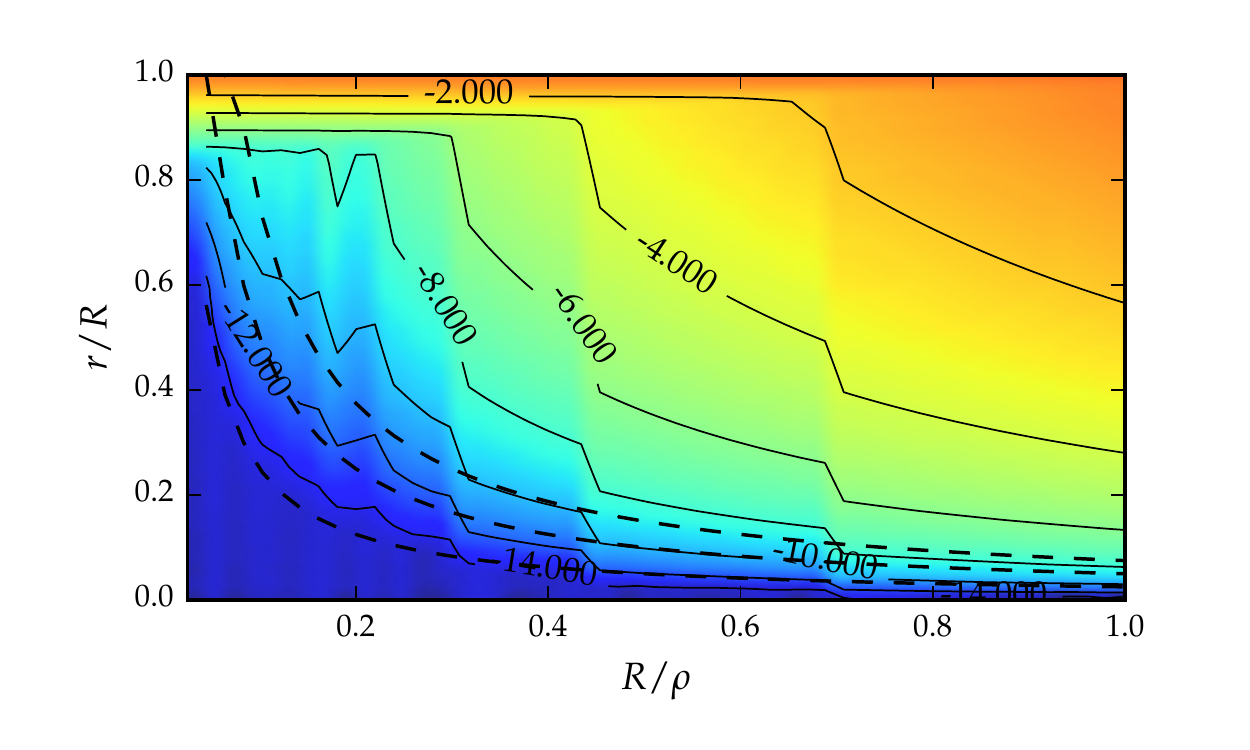}
  \mcaption{fig:param-reval}{Error at different evaluation radii}{ The
    error for evaluation of a single expansion with various $\rp$ and
    $\re$, but fixed $\rc = \rho/40$ and $\rho$.
    The expansion is interpolating a harmonic function (similar to the
    one used in \cref{fig:err-bound}) with singularity at distance
    $\rho=4$, using the Laplace \dl kernel.
    The dotted lines are $\re = m \rc$ for $m=1,2,\text{ and } 3$.
    In practice, we have no direct control on $\frac{\rp}{\rho}$,
    and it is implied by the panel size.
    Here we chose $\np=64$, and $\nc=2\np$; the trends are the same
    for lower $\np$ and $\nc$.
  }
\end{figure}

\para{Choice of parameters} Using  the model \cref{eqn:err-bound}, one can make choices for $\rp$, $\rc$,
$\remax$, and $\np$ to achieve a desired accuracy $\acc$.
An unknown in applying this in a practical setting
is the singularity distance $\rho$.
However, in any high-accuracy choice of boundary quadrature, such as
the adaptive panel quadrature of \cref{sec:formulation}, panels are
refined such that the data $f$ and hence the density $\phi$ and the
solution $\u$ are smooth on the local panel scale $L$, thus we expect
singularities to be at least of order $L$ distant from the center.
Indeed, we experimentally observe (in tests where we know the location of
singularity, e.g., \cref{fig:center-dist} or \cref{s:bvp}) that when the panels are adaptively refined, $L<\rho$,
and consequently the convergence behavior is most like the left-hand
plot of \cref{fig:err-bound}.

Given the target accuracy of $\acc$ for the solution and the selected
native quadrature order $q$, the adaptive refinement of boundary sets
the panel length $L$. We use the following steps to glean the value of
other parameters. Since the constants in the error estimates are
problem dependent and unknown, we set them to unity. To have a
concrete example, we pick $\acc=\sci{-10}$ and $q=16$.
\begin{enumerate}[\quad(1)]
\item \lhb{Setting $\remax$:} By construction, points farther than
  $2\remax$ from the boundary are evaluated using the native
  quadrature. To meet the desired error $\acc$ at these points,
  $\frac{L}{\remax} \approx 8 \acc^{1/2q}$, which implies
  $\remax\approx L/4$ for $\acc=\sci{-10}, q=16$.
\item \lhb{Setting $k_m$, $\rp/\rc$, and $\np$:} Requiring that the
  two terms in the error estimate (i.e., proxy point representation
  and extrapolation errors) have similar contribution at the on
  surface point ($\re=\remax$) and assuming that $L \approx \rho$ we
  can estimate the minimum required $k$ based on the proxy
  representation error in the rough regime:
  \begin{align}
    \left(\dfrac{\remax}{\rho}\right)^{k/2} \approx \acc \quad
    \text{or} \quad k \approx \dfrac{2\log\acc}{\log{(\remax/L)}}~,
  \end{align}
  implying $k\approx 32$ for $L/\delta=4, \acc=\sci{-10}$.  Since $k$
  is bounded by $\min(k_m,\np)$, knowing minimum $k$ implies a lower
  bound for $k_m$ and $\np$. Therefore, reorganizing \cref{k}, we have
  ${\rp}/{\rc} = \apinv^{2/k} \approx 7$, for $\apinv=\sci{-14}$.

\item \lhb{Setting $\rc/\remax$ and $\beta$:} Inspecting the
  extrapolation error at an on surface point, we have
  \begin{align}
    e_e(\remax) \approx \ech \left(\dfrac{\remax}{\rc}\right)^{k/2} \approx
    \left(\dfrac{L}{4 \beta (\remax - \rc)}\right)^{2q}
    \left(\dfrac{\remax}{\rc}\right)^{k/2} \approx \left(\dfrac{L}{4 \beta
      \remax}\right)^{2q} \dfrac{1}{(1-\theta)^{2q}\theta^{k/2}}~,
  \end{align}
where $\theta = \rc/\remax$. This expression attains its minimum at
$\theta = \frac{k}{4q+k}$. For $q=16$ and $k=32$, we have $\theta =
1/3$. As we require that two terms in the error estimate have similar
contribution, we use $e_e(\remax)$ and estimate $\beta$:
\begin{align}
  \beta \approx \frac{L/4\remax}{(1-\theta)\theta^{k/4q} \acc^{1/2q}}~,
\end{align}
implying $\beta = 5$, for the choices of parameter listed above.
\end{enumerate}





Note that we have not analyzed the effect of finite $\nc$,
but find that the choice $\nc=2\np$ behaves indistinguishably
from the limit $\nc\to\infty$; we attribute this to the rapid
convergence of the periodic trapezoid rule on the check points.
\begin{figure}[!bt]
  \centering
  \small
  \setlength\figureheight{2.2in}
  \setlength\figurewidth{2.2in}
  \includepgf{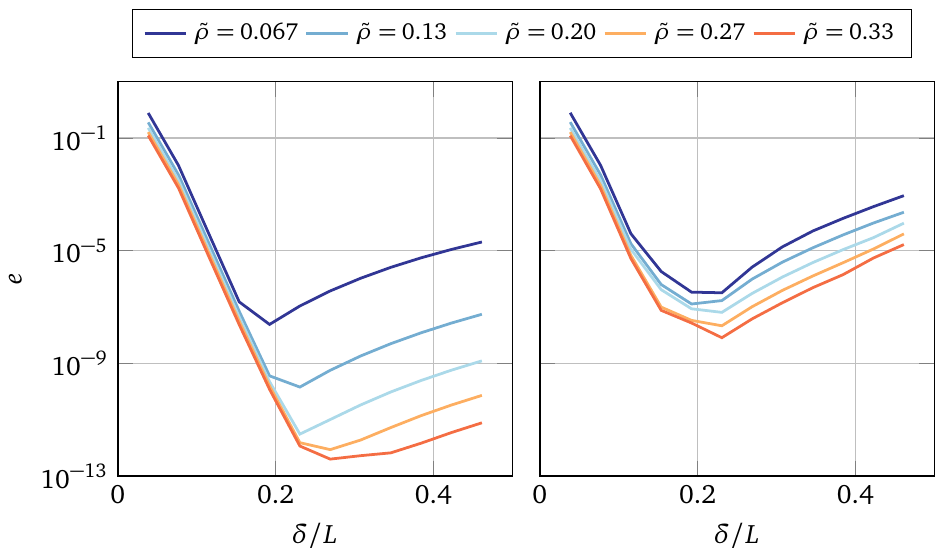}
  \mcaption{fig:center-dist}{Error vs. center and singularity
    distances}{The induced error for singularities and centers at
    various distances from the boundary for the Laplace Dirichlet
    interior \abbrev{BVP}, in the domain shown in \cref{fig:inf-error}.
    The boundary data is generated by putting a Laplace singularity at
    distance $\tilde\rho$ from the boundary\emdash/the singularity
    distance to the center of expansion is $\rho \ge \tilde \rho +
    \remax$.
    The density is solved directly and \qbkix is used only for evaluation.
    The error is computed using the known solution corresponding to
    the boundary data.
    The left plot shows the errors for the case with fixed number of
    panels on the boundary ($M=40$ panels).
    In this plot, because $L$ is fixed, $L/\tilde\rho$ is decreasing
    by increasing $\tilde\rho$.
    The right plot shows the errors for adaptive refinement of the
    boundary with $\atol=\sci{-11}$.
    Here, since $L$ is chosen adaptively due to the boundary data, it
    increases as the solution becomes smoother.
    Because, $L$ is chosen proportional to $\tilde\rho$, the error
    curves almost collapse to one.
    We use $\np=64$, $\nc=2\np$, $\rc=\remax/3$ and $\rp=8\rc$.
    In both cases, the center of expansion is located based on the
    panel size at distance $\remax$.
  }
\end{figure}

\section{Numerical experiments\label{sec:results}}
In this section, we present the results of numerical tests demonstrating the
accuracy and versatility of the \qbkix algorithm for
on-surface evaluation needed for the boundary integral equation solver
and  solution evaluation close to the surface.
In the following experiments, unless noted otherwise, we use \qbkix
for both tasks.

\subsection{Convergence with respect to the number of panels}
In \cref{tbl:m-conv}, we report the convergence of the solution
evaluated at the interior points using non-adaptive boundary
quadrature with increasing number of panels.
The test solution is the potential due to a  set of singularities
at the source points shown outside the domain.
These source points are used to generate the boundary data $f$ and the
reference solution to check the error.
For all problems, the \dl formulation is used, except for the
Helmholtz for which a combined-field formulation $u = \conv{D}[\phi] +\ii
\omega \conv{S}[\phi]$, where $\conv{S}$ is the single-layer potential
\cite[Section~3.2]{coltonkress}, is used. This representation addresses
problems associated with resonance of the complementary domain.
The \dl (or combined-field) density $\phi$ is solved using \qbkix to
evaluate the matrix-vector product in each iteration of \gmres.
The error in the density is quantified by computing the solution from $\phi$, \cref{eqn:generic-bi-sol}, at a set of target points in the interior of the domain.
For the first three kernels, which are smooth, we also report the
convergence using the \nystrom
(\emph{direct}) evaluation, \cref{Anyst}, which by comparison against one- or
two-sided \qbkix shows how much of the error is due to \qbkix.

In all cases, it can be seen that \qbkix gives high-order convergence
rate that is independent of the type of the kernel.
We notice that the error performance of the two-sided variant is
worse than one-sided at the same number of panels (however,
as we discuss below, it is valuable since it improves the convergence
rate of \gmres).

\begin{table}[!bt]
  \newcolumntype{C}{>{\centering\let\newline\\\arraybackslash\hspace{0pt}$}c<{$}}
  \newcolumntype{R}{>{\centering\let\newline\\\arraybackslash\hspace{0pt}$}r<{$}}
  \newcolumntype{L}{>{\centering\let\newline\\\arraybackslash\hspace{0pt}$}l<{$}}
  \newcolumntype{A}[1]{>{\centering\let\newline\\\arraybackslash\hspace{0pt}}D{.}{.}{#1}}
  %
  \centering
  \scriptsize
  \setlength{\tabcolsep}{4pt}
  \begin{tabular}{c l l CCCC}\toprule
    Geometry                                               & Kernel                                    & Quadrature   & \multicolumn{4}{c}{Absolute error (Number of panels)}                              \\\midrule
    \multirow{3}{*}{\includegraphics[height=.35in,width=.35in]{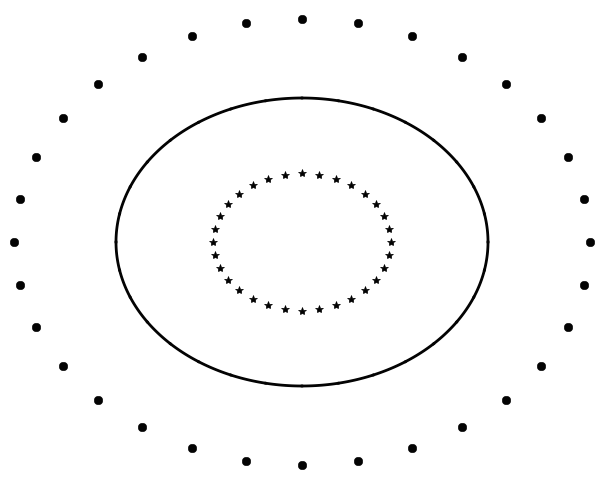}}  & \multirow{3}{*}{Laplace}                  & Direct       & 2.90\e{-06}~(2)  & 9.46\e{-10}~(4)  & 6.42\e{-14}~(6)  & 1.98\e{-14}~(8)  \\
                                                           &                                           & \qbkix (one) & 3.39\e{-06}~(2)  & 9.69\e{-10}~(4)  & 4.46\e{-12}~(6)  & 3.54\e{-12}~(8)  \\
                                                           &                                           & \qbkix (two) & 2.25\e{-05}~(2)  & 4.07\e{-07}~(4)  & 2.24\e{-08}~(6)  & 2.37\e{-09}~(8)  \\\midrule
    \multirow{14}{*}{\includegraphics[height=1in,width=1in]{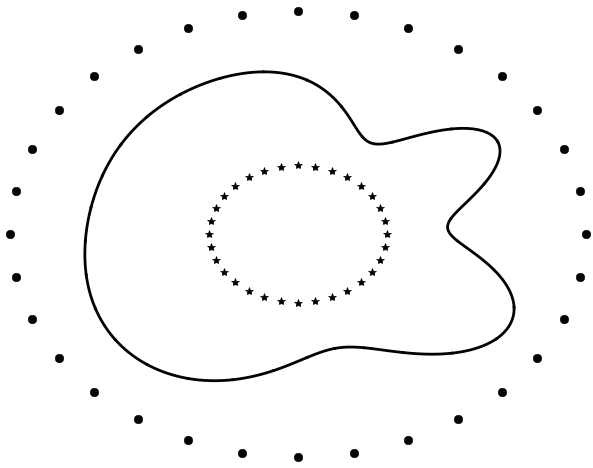}} & \multirow{3}{*}{Laplace}       & Direct       & 5.80\e{-07}~(6)  & 8.52\e{-07}~(8)  & 1.67\e{-09}~(10) & 5.65\e{-12}~(12) \\
                                                           &                                           & \qbkix (one) & 2.49\e{+00}~(6)\ensuremath{^1}  & 1.32\e{-04}~(8)  & 4.62\e{-09}~(10) & 3.09\e{-09}~(12) \\
                                                           &                                           & \qbkix (two) & 4.29\e{-01}~(6)  & 3.06\e{-04}~(8)  & 4.25\e{-07}~(10) & 1.50\e{-07}~(12) \\\cmidrule(l){2-7}
                                                           & \multirow{5}{*}{Stokes}                   & Direct       & 1.48\e{-04}~(6)  & 6.67\e{-05}~(8)  & 6.51\e{-08}~(10) & 6.06\e{-10}~(12) \\
                                                           &                                           & \qbkix (one) & 2.89\e{-08}~(20) & 4.78\e{-09}~(24) & 1.73\e{-09}~(28) & 6.38\e{-10}~(32) \\
                                                           &                                           & \qbkix (two) & 6.95\e{-06}~(16) & 4.87\e{-08}~(32) & 3.31\e{-09}~(48) & 9.45\e{-10}~(64) \\\cmidrule(l){2-7}
                                                           & \multirow{2}{*}{Helmholtz$^{2}$ ($\omega=2$)}   & \qbkix (one) & 2.12\e{-04}~(8)  & 1.20\e{-09}~(12) & 4.22\e{-10}~(16) & 2.09\e{-11}~(20) \\
                                                           &                                           & \qbkix (two) & 3.97\e{-04}~(8)  & 1.91\e{-07}~(12) & 3.42\e{-08}~(16) & 7.92\e{-09}~(20) \\\cmidrule(l){2-7}
                                                           & \multirow{2}{*}{Yukawa  ($\lambda=2$)}    & \qbkix (one) & 1.60\e{-04}~(8)  & 6.42\e{-07}~(12) & 3.84\e{-09}~(16) & 1.48\e{-09}~(20) \\
                                                           &                                           & \qbkix (two) & 5.44\e{-04}~(8)  & 1.27\e{-07}~(12) & 2.19\e{-08}~(16) & 4.79\e{-09}~(20) \\\cmidrule(l){2-7}
                                                           & \multirow{2}{*}{Elastostatic ($\nu=0.1$)} & \qbkix (one) & 2.07\e{-03}~(8)  & 7.16\e{-06}~(12) & 4.35\e{-07}~(16) & 7.19\e{-07}~(20) \\
                                                           &                                           & \qbkix (two) & 3.17\e{-02}~(8)  & 1.27\e{-05}~(12) & 2.26\e{-06}~(16) & 6.77\e{-07}~(20) \\
    \bottomrule
    \multicolumn{7}{l}{\parbox{.95\linewidth}{
        \vspace{.5em}
        \addtocounter{footnote}{1}
        \thefootnote~~When there are a few panels on the boundary, a
        check circle  may be placed near other panels which adversely
        affects the error.
        \newline
        \addtocounter{footnote}{1}
        \thefootnote~~For Helmholtz equation, we use a combined field formulation.
    }}
  \end{tabular}
  \mcaption{tbl:m-conv}{Solution convergence vs. number of panels}{
    Error in the solution to interior Dirichlet boundary value
    problems using non-adaptive $M$-panel quadrature and \qbkix for
    solution.
    The subplots show $\Gamma$ (solid) and the exterior sources used
    to generate the solution, and interior test points.
    There are 40 source points outside the domain and error is
    measured on 40 points inside.
    The error is the maximum of absolute error over these interior
    points.
    The numerical parameters are $\np=32$, $\nc=2\np$, $\rp=8\rc$, and
    $\remax=3\rc$.
    ``Direct'' indicates usage of the quadrature of \cref{Anyst}
    instead of \qbkix for the linear solve. ``One'' and ``two''
    indicate one- or two-sided versions of on-surface \qbkix  discussed in \cref{s:sided}.
  }
\end{table}

\subsection{Operator spectrum and \gmres convergence rate\label{ssc:spectrum}}
We now perform numerical tests of the one-sided and two-sided variants
of on-surface evaluation of \qbkix discussed in \cref{s:sided} and
compare it to direct use of an accurate quadrature. To simplify
comparisons, we use an operator with a smooth kernel (Laplace). The
spectra and convergence behavior for singular kernels is similar.
In \cref{fig:spectrum} we plot\emdash/for the domain shown in
\cref{fig:inf-error} and the Laplace equation\emdash/ the eigenvalues
for four different approximations to the operator $-\half + D$:
one-sided (interior) \qbkix, the one-sided (exterior) \qbkix,
two-sided \qbkix, and the quadrature given by
\cref{Anyst}, to which we refer as \emph{direct}.
The exterior version of \qbkix is constructed similarly to the
interior variant discussed in \cref{sec:formulation}.
The only modification is that for each collocation point $\vx_0$ on
$\Gamma$, we place an expansion center at $\cc = \vx_0 + \delta \vn$.
We see that the one-sided variants have clusters of eigenvalues
near zero, whereas the two-sided variant and the \nystrom
matrix have a cleaner  spectrum with eigenvalue clustering only around $\frac{1}{2}$.

A broader spread of the eigenvalues has a negative impact on
\gmres convergence \cite{nachtigal1992a}. \cref{fig:spectrum}, right,
shows \gmres residual versus the iteration number for
the interior, two-sided, and direct operators with two different
right-hand sides (boundary data corresponding to a harmonic function and a
random right-hand side).

The convergence of one-sided interior \qbkix is identical to the
\nystrom method convergence up to the residual magnitude on the order
of numerical accuracy of \qbkix, but it slows down once the residual
decreases below this value (near \sci{-9}).
The two-sided variant has identical convergence behavior to the
direct method, and converges in a few iterations.
We also show the residual for a random-right hand side to expose the
effect of near-zero eigenvalues: we see that convergence is very slow
for the one-sided scheme in this case, but for the two-sided scheme it is the same as for
the true smooth data $f$.

\begin{figure}[!bt]
  \centering
  \small
  \setlength\figureheight{1.3in}
  \setlength\figurewidth{2in}
  \includepgf{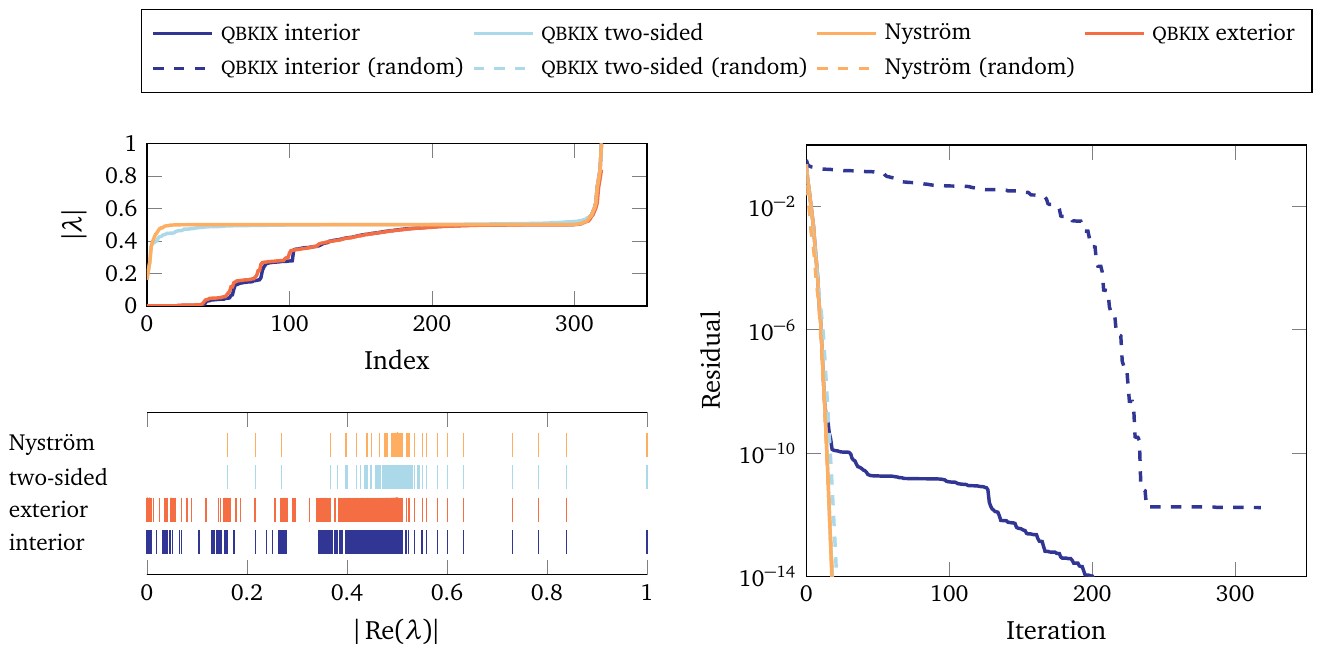}
    \mcaption{fig:spectrum}{The spectra of discretizations of
    the Laplace \dl operator}{This figure shows eigenvalues, and the \gmres
    convergence rate, for different discretizations of the Laplace \dl
    operator in the domain shown in \cref{fig:inf-error}.
    The left plots show the real part and the magnitude of the
    eigenvalues corresponding the one-sided interior \qbkix, one-sided
    exterior \qbkix, two-sided \qbkix, and the plain \nystrom matrix.
    See \cref{s:sided,ssc:spectrum}.
    The right plot shows the residual versus the iteration number for
    the three interior variants with two different right hand sides (boundary
    data corresponding to a harmonic function or random data).
    The residual of the two-sided and \nystrom schemes are indistinguishable.
    }
\end{figure}

\subsection{Error for Dirichlet problems for five {\pde}s\label{s:bvp}}
For this set of tests, we use adaptive refinement as described in
\cref{sec:formulation}.
We use \qbkix both as the on-surface quadrature scheme when
solving for the desired density as well as the evaluator for the
near-singular integrals.
As before, we use boundary data sampled from a sum of fundamental
solutions centered at a set of points close to the boundary.
\cref{fig:inf-error} plots the error across the domain for all of
the~{\pde}s listed in \cref{eqn:pde-type}, on the points lying on a
$600\times 600$ grid and interior to the domain.
When an evaluation point is within $2\delta$ distance from the
boundary, it is evaluated using the nearest \qbkix expansion.
The remaining points are evaluated using \cref{eqn:smooth-quad}
applied to \cref{eqn:generic-bi-sol}.

\begin{figure}[!tb]
  \centering
  \setlength\figureheight{1.8in}
  \subfloat[Laplace $(M=30)$\label{sfg:laplace}]{\includegraphics[height=\figureheight]{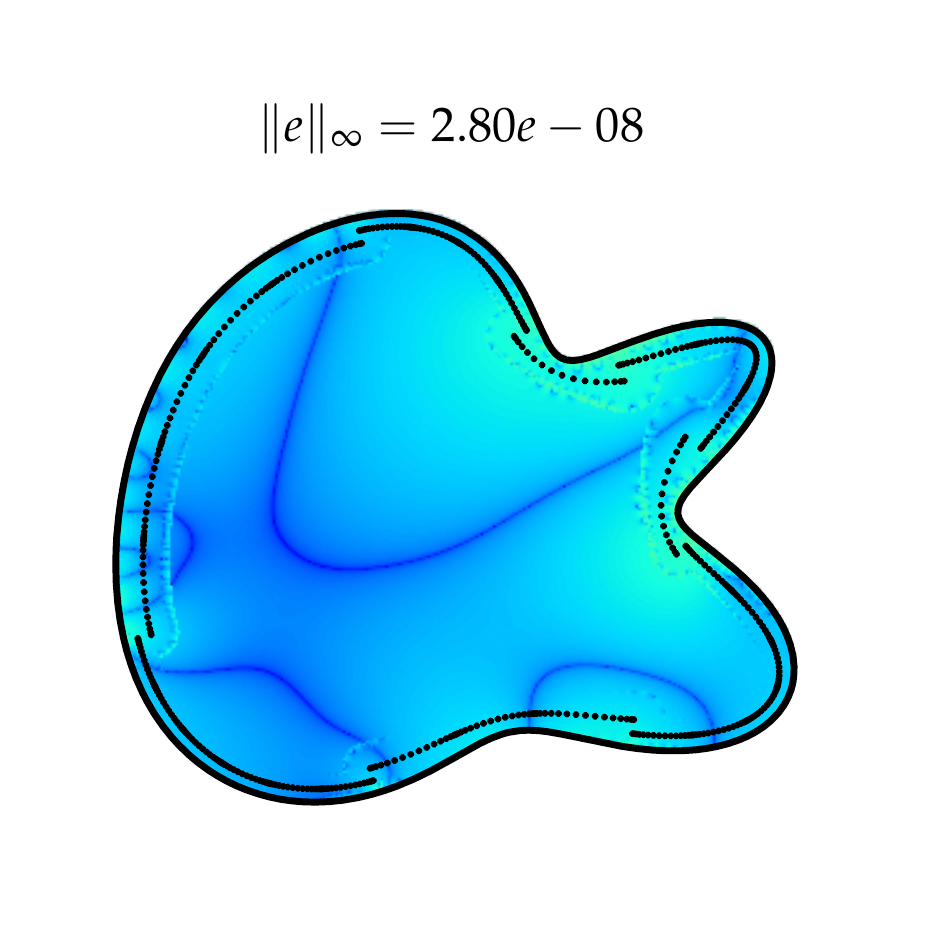}}
  \subfloat[Helmholtz $(M=30)$\label{sfg:helmholtz}]{\includegraphics[height=\figureheight]{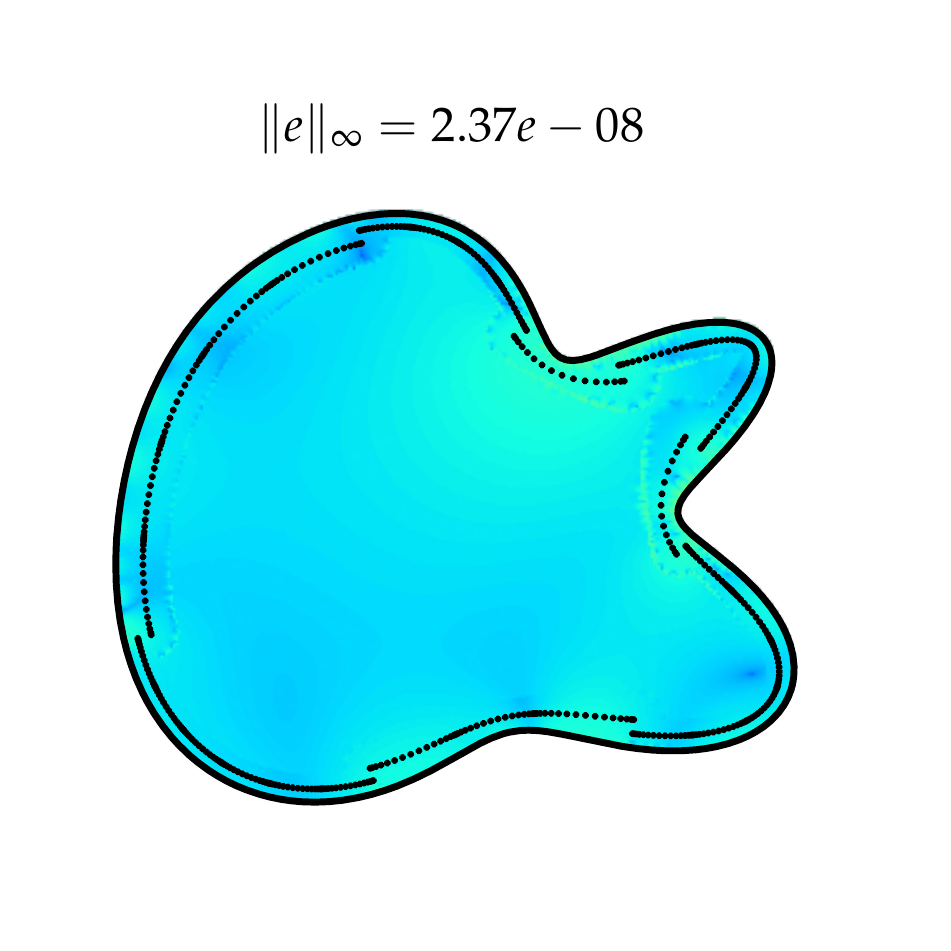}}
  \subfloat[Yukawa $(M=30)$\label{sfg:yukawa}]{\includegraphics[height=\figureheight]{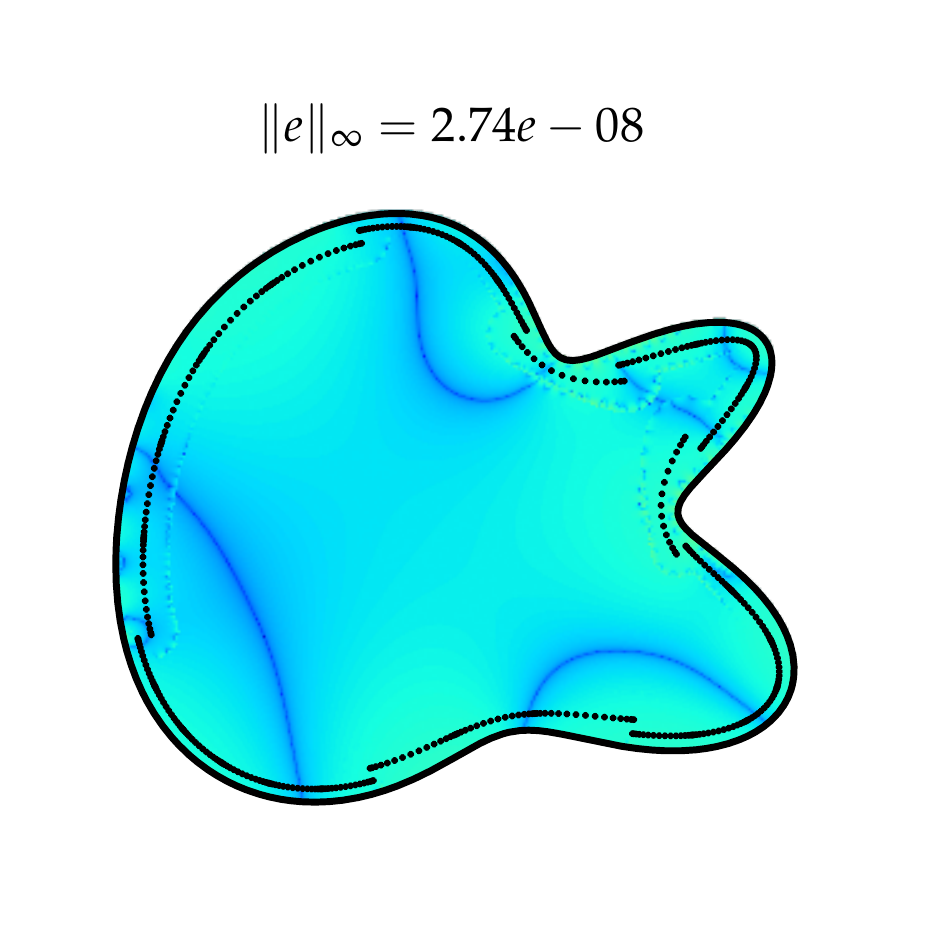}
    \includegraphics[height=.95\figureheight]{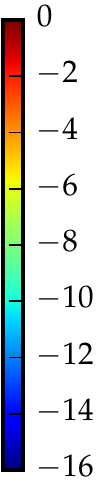}}\\
  \subfloat[Stokes velocity $(M=48)$\label{sfg:stokes}]{\includegraphics[height=\figureheight]{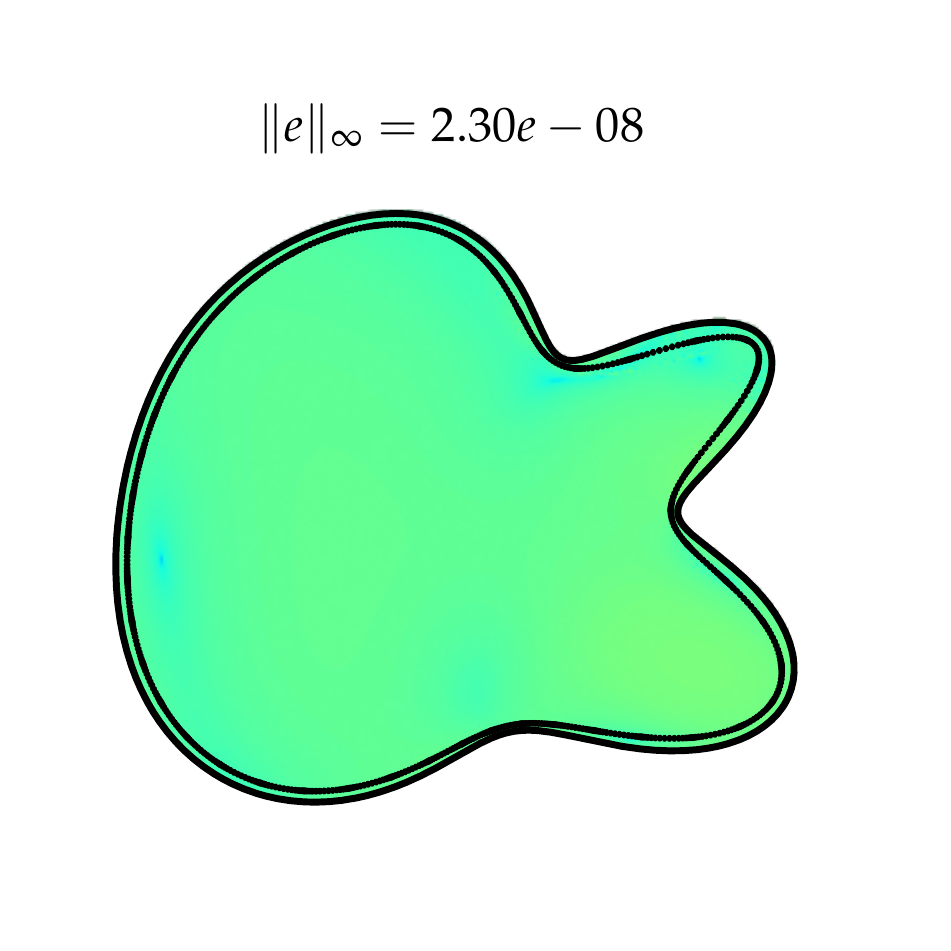}}
  \subfloat[Elastostatic $(M=44)$\label{sfg:elastostatic}]{\includegraphics[height=\figureheight]{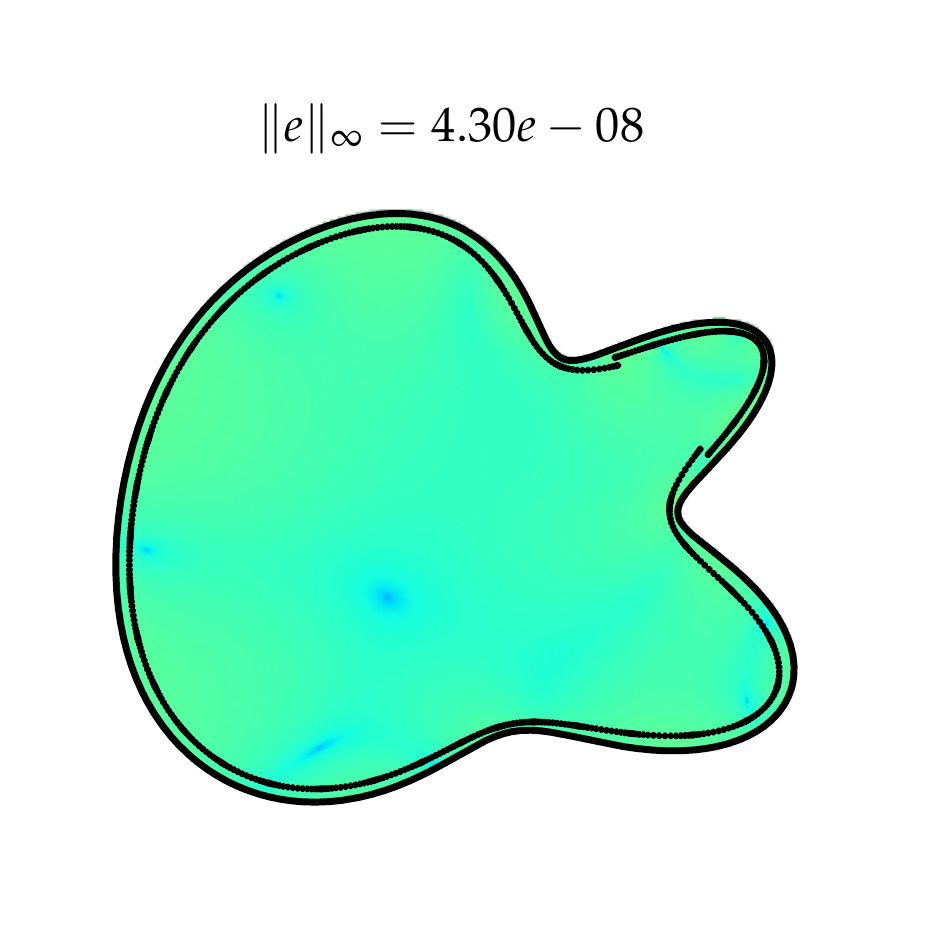}
    \includegraphics[height=.95\figureheight]{interior-cbar}}\\
  \subfloat[Smooth stokes velocity $(M=26)$\label{sfg:stokes-velocity}]{\includegraphics[height=\figureheight]{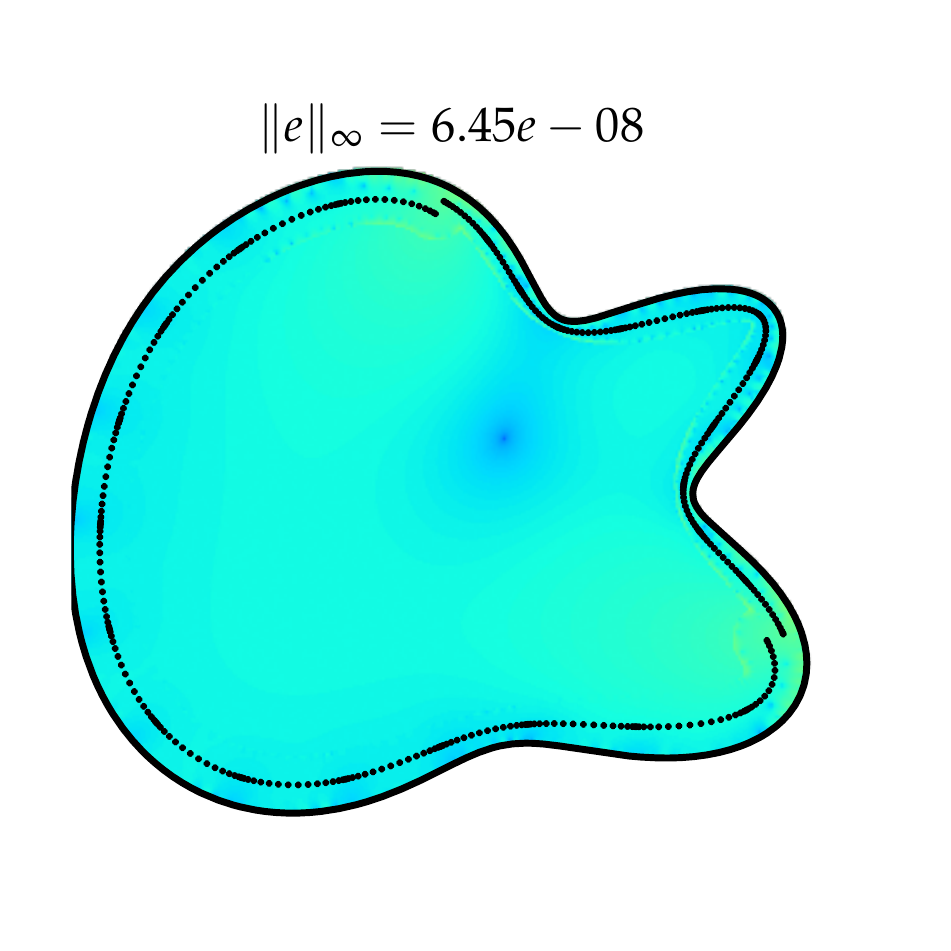}}
  \subfloat[Smooth stokes pressure $(M=26)$\label{sfg:stokes-pressure}]{\includegraphics[height=\figureheight]{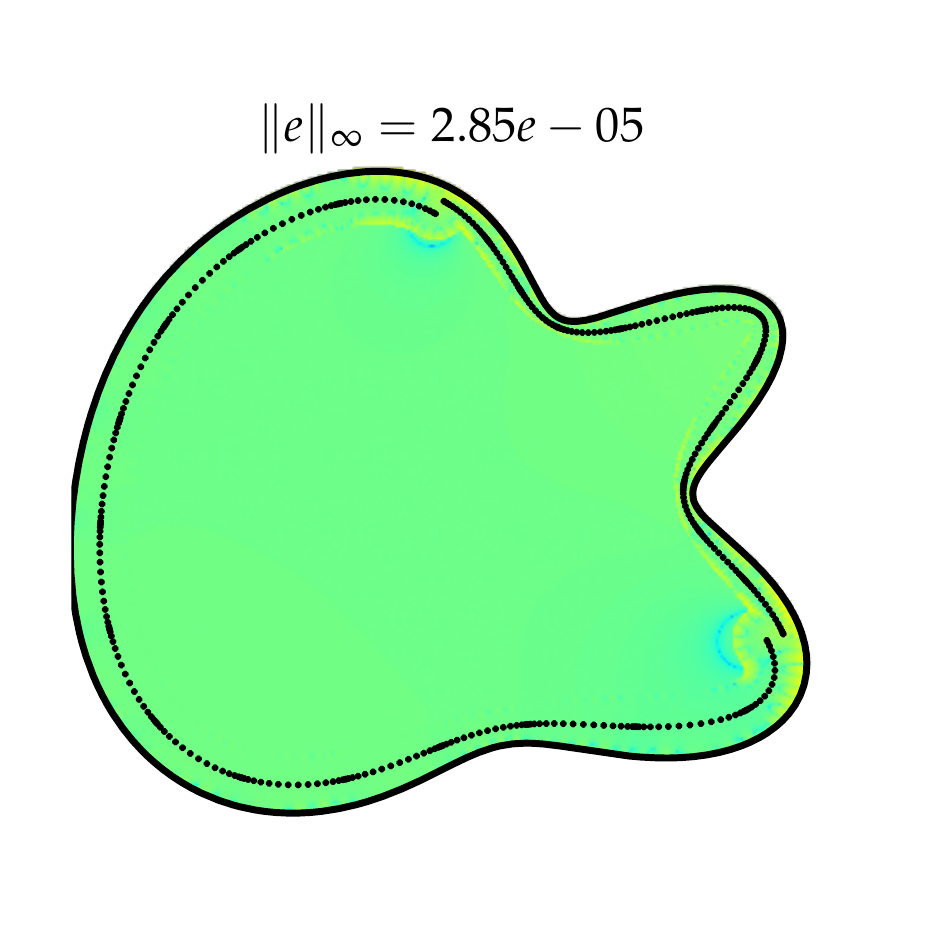}
    \includegraphics[height=.95\figureheight]{interior-cbar}}
  \mcaption{fig:inf-error}{The $\log_{10}$ of pointwise error}{ The
    interior Dirichlet boundary value problem is solved with known
    solution generated by source points distributed over an exterior
    circle as shown in the lower figure in \cref{tbl:m-conv}, apart
    from in \subref{sfg:stokes-velocity} and
    \subref{sfg:stokes-pressure} where we use the cubic flow with velocity
    $\vu=[y^3,x^3]$ and pressure $p=6xy$.
    Error is evaluated on the same fine grid used for visualization
    ($600\times 600$).
    We use $q=16$ node \gl panels and set $\atol=10^{-13}$ in the
    adaptive panel quadrature set-up.
    $M$ denotes the number of boundary panels.
    The expansion centers $\cen$ are shown by black dots close to the
    boundary.
  }
\end{figure}

We observe that parameter choices which were selected for the Laplace
equation perform well for the other {\pde}s.
As expected, the highest error is due to expansions for panels
adjacent to larger ones (e.g. \cref{sfg:laplace}).

\subsection{Domain with a large number of corners\label{ssc:corner}}
As a final example, we use \qbkix in a domain with 256 corners as
shown in \cref{fig:corner}.
A Laplace boundary value problem is solved using \gmres with tolerance
for relative residual set to $\rtol = \sci{-6}$.
The boundary condition is generated similar to the examples in
\cref{s:bvp} by placing 32 source points on a circle with radius 0.75
centered at $[0.5,0.5]$ (the domain's bounding box is
$[0,1]\times[0,1]$).

The boundary of the domain is adaptively refined, with minimum panel length set to $\acc_l=\rtol/10$.
Large panels are also refined based on the adaptive criterion we
outlined in \cref{sec:formulation}.
The dyadic and adaptive refinements result in a total of 9560 panels.


Due to the singularities on the boundary, the system matrix is
ill-conditioned.
The ill-conditioning is greatly reduced using left and right
preconditioners with square root of smooth quadrature weights on its diagonal
\cite{bremer2012}, solving for density in $L^2$ sense.
Considering this preconditioning and since the last panel in each side of
the corner is of length smaller than $\rtol/10$, we set the density on
those panels to zero (effectively deleting the last two panels).
The \gmres converges after 33 iterations;  we use \kifmm (with
accuracy set to $\rtol/10$) for fast evaluation.

\begin{figure}[!tb]
  \centering
  \setlength\figureheight{2in}
  \setlength\figurewidth{2in}
  \subfloat[The $\log_{10}$ of pointwise error]{\includegraphics[width=\figurewidth]{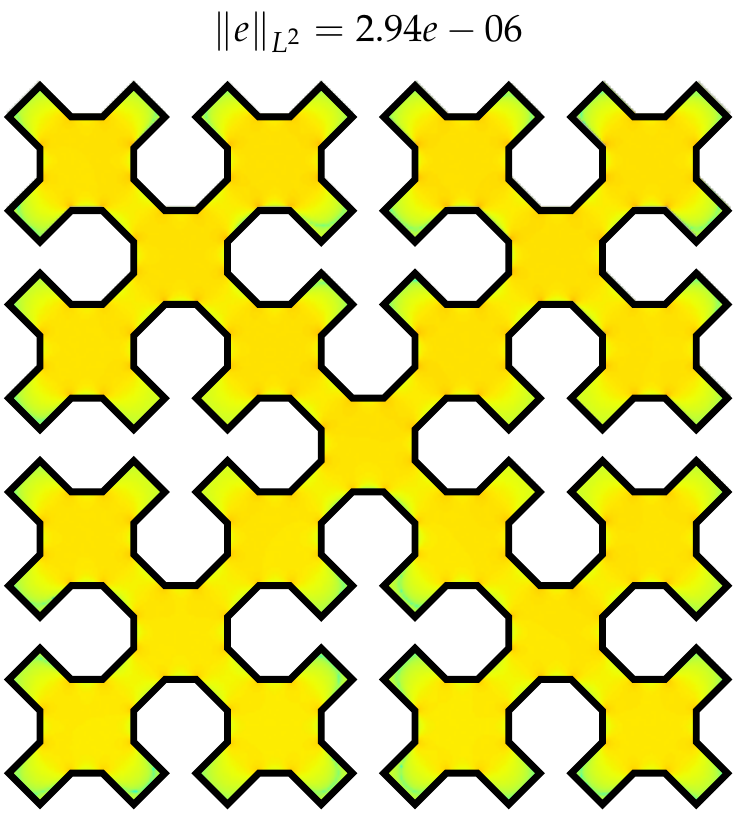}
    \hspace{1em}\includegraphics[height=\figureheight]{interior-cbar}}
  \hspace{.15\figurewidth}
  \subfloat[The solution]{\includegraphics[width=\figurewidth]{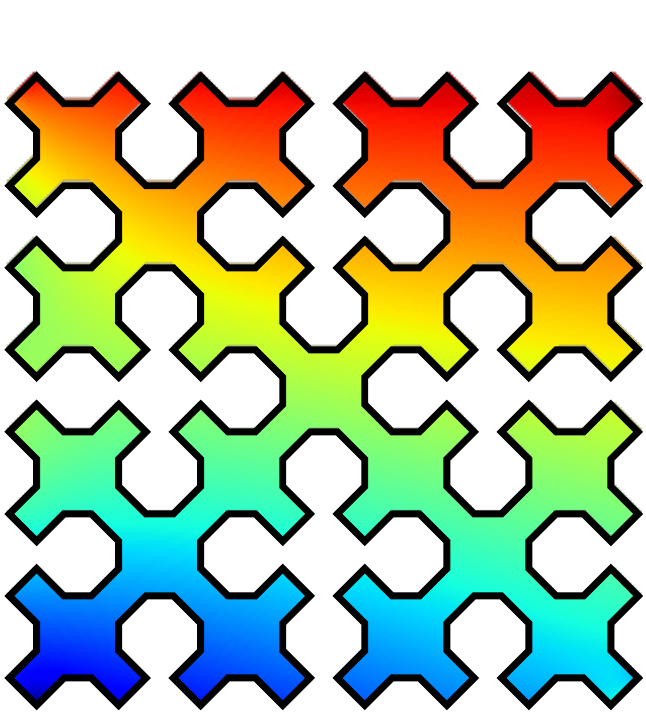}
    \hspace{1em}\includegraphics[height=\figureheight]{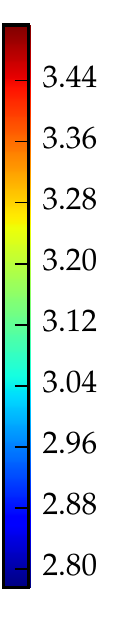}}
  \mcaption{fig:corner}{\qbkix in a domain with 256 corners}{}
\end{figure}

\section{Conclusions\label{sec:conclusion}}
In this paper we introduced a new quadrature scheme for the high-order
accurate evaluation of layer potentials associated with general
elliptic \pde  on the domain boundary and close to it.
The scheme\emdash/which builds local solution approximations using
a refined evaluation and the solution of small linear systems\emdash/
relies solely on the evaluation of the underlying kernel, so is
essentially \pde-independent.
It is highly flexible, being agnostic as to the boundary condition
type, the layer representation, and crucially, the dimension of the
problem.
We have analyzed the eror behavior of the scheme for Laplace and
Helmholtz cases.
It also fits naturally in the framework of existing fast
kernel-independent algorithms for potential evaluation such as the
\kifmm, as it uses similar local approximations.

We have tested its accuracy for three scalar- and two vector-valued
\twod Dirichlet boundary-value problems that are common in engineering
problems.
We have not attempted to optimize performance, and leave that
for future work.

There are several obvious extensions that have motivated this initial
study that we plan to pursue:
\begin{enumerate}[(1)]
\item Generalization to \threed. High-order singular quadratures for surfaces are
  complicated, application dependent, and scarce.
  Since it requires only pointwise kernel evaluations, \qbkix is by
  design very easy to implement in \threed using proxy and check
  surfaces, and would handle a wide class of {\pde}s.
  The constants will be larger, but the linear systems (anticipated to
  be of size around $10^3$) would still be very practical.
\item Generalization to other boundary conditions.  \qbx, and thus
  also \qbkix, can apply without modification, for instance, the
  normal derivative of the double-layer operator, which is
  hypersingular.
\item Integration with \kifmm. In this work, we only used \ki
  \abbrevformat{FMM} for fast evaluation of potential on the check
  points.
  However, we expect performance gains by reusing the local expansion
  of \kifmm as a \qbkix expansion.
\item Local \qbkix.
  The construction of local schemes which automatically handle general
  domains with thin features (i.e., with geodesically distant parts of
  the boundary in close proximity in space) without excessive
  refinement needed for the panel size to be on the order of feature
  size,  is important for making  the method practical.
  \cite{ce} proposed the \emph{local} version of  \qbx, in which only
  the contribution of the nearby panels to a target is evaluated using
  expansions, while contributions of more distant panels is evaluated
  using standard quadrature.  Implementing this idea is nontrivial
  however, as the end-points of the group of neighboring  panels
  produce new singularities that can affect the convergence rate.
\item Generalization of analysis to all kernels.
  As Remark~\ref{r:anal} discusses, this is a nontrivial missing piece
  in the theoretical foundations.
\end{enumerate}




\begin{acknowledgements}
  We extend our thanks to Manas Rachh, Andreas Kl{\"o}ckner, Michael
  O'Neil, and Leslie Greengard for stimulating conversations about
  various aspects of this work.
  A.R. and D.Z. acknowledge the support of the US National Science
  Foundation (NSF) through grant DMS-1320621; A.B. acknowledges the
  support of the NSF through grant DMS-1216656.
\end{acknowledgements}

\begin{appendices}
  \crefalias{section}{apx}
  \section{List of kernels\label{apx:kernels}}
Here we list the kernels for the single- and double-layer potentials
for the {\pde}s considered text.
In each case $\vx$ and $\vy$ are points in $\R^2$ and $\vr \defeq
\vx-\vy$.
The single-layer kernel is the fundamental solution.
In \dl kernels, $\vn$ is the unit vector denoting the dipole
direction, which in the context of boundary integral formulation is
the outward pointing normal to the surface.

\begin{itemize}[\quad$\bullet$]
\item Laplace:
  \begin{alignln}
    \Delta \u                   & = 0,                                              \\
    S(\vx,\vy)                  & = -\frac{1}{2\pi}\log |\vr|,                      \\
    D(\vx,\vy)                  & = \frac{1}{2\pi}\frac{\inner{\vr}{\vn}}{|\vr|^2}, \\
    \lim_{\vy\to\vx} D(\vx,\vy) & = -\frac{\kappa}{4\pi}, \qquad \vx,\vy\in\Gamma, \quad\text{(where $\kappa$ is the signed curvature)}.
  \end{alignln}

\item Yukawa:
  \begin{alignln}
    \Delta \u - \lambda^2 \u & = 0,                                \\
    S(\vx,\vy)               & = \frac{1}{2\pi}K_0(\lambda |\vr|), \\
    D(\vx,\vy)               & = \frac{\lambda}{2\pi}\frac{\inner{\vr}{\vn}}{|\vr|}K_1(\lambda|\vr|),
  \end{alignln}
  where $K_0, K_1$ are modified Bessel functions of the second kind of order zero and one, respectively.

\item Helmholtz:
  \begin{alignln}
    \Delta \u + \omega^2 \u & = 0,                                \\
    S(\vx,\vy)              & = \frac{\ii}{4}H_0^1(\omega |\vr|), \\
    D(\vx,\vy)              & = \frac{\ii\omega}{4}\frac{\inner{\vr}{\vn}}{|\vr|} H_1^1(\omega|\vr|),
  \end{alignln}
  where $H^1_0, H^1_1$ are respectively modified Hankel functions of the first kind of order zero and one.
\item Stokes:
  \begin{alignln}
    -\Delta \vu + \nabla \scalar{p} & = 0, \qquad \Div \vu = 0,                                                 \\
    S(\vx,\vy)                      & = \frac{1}{4\pi}\left(-\log|\vr| + \frac{\vr \otimes \vr}{|\vr|^2}\right), \\
    D(\vx,\vy)                      & = \frac{\inner{\vr}{\vn}}{\pi}\frac{\vr \otimes \vr}{|\vr|^4},            \\
    \lim_{\vy\to\vx} D(\vx,\vy)     & = -\frac{\kappa}{2\pi} \vector{t}\otimes\vector{t},                       \\
    P(\vx,\vy)                      & = -\frac{1}{\pi|\vr|^2} \left( 1 - 2\frac{\vr \otimes \vr}{|\vr|^2} \right) \vn.
  \end{alignln}

\item Navier: Linear elasticity for isotropic material with shear modulus $\mu$ and Poisson ratio $\nu$,
  \begin{alignln}
    \mu\Delta \vu + \frac{\mu}{1-2\nu} \Grad\Div \vu & = 0,                                                                                       \\
    S(\vx,\vy)                                  & = -\frac{3-4\nu}{8\pi(1-\nu)}\log|\vr| + \frac{1}{8\pi(1-\nu)} \frac{\vr \otimes \vr}{|\vr|^2}, \\
    D(\vx,\vy)                                  & = \frac{1-2\nu}{4\pi(1-\nu)}\left(\frac{\inner{\vr}{\vn} + \vn \otimes \vr - \vr \otimes \vn}{|\vr|^2} + \frac{2}{1-2\nu} \frac{\inner{\vr}{\vn}\, \vr \otimes \vr}{|\vr|^4}\right).
  \end{alignln}
\end{itemize}

\end{appendices}

\bibliographystyle{spmpsci} 
\bibliography{refs}

\end{document}